%% file: neurips_2021.tex
\documentclass{article}

    \usepackage[final, nonatbib]{neurips_2021}

\usepackage[utf8]{inputenc} 
\usepackage[T1]{fontenc}    
\usepackage{hyperref}       
\usepackage{url}            
\usepackage{booktabs}       
\usepackage{amsfonts}       
\usepackage{nicefrac}       
\usepackage{microtype}      
\usepackage{xcolor}         

\usepackage{graphicx}
\usepackage{amsmath,amssymb}
\usepackage{array}
\usepackage{enumitem}
\usepackage{xcolor}
\usepackage{bm}
\usepackage{multirow}
\usepackage{mathdots}
\usepackage{subfigure}
\usepackage{amsthm}
\usepackage{tikz}
\usepackage[bottom]{footmisc}

\usepackage{algorithm,algpseudocode}

\newtheorem{theorem}{Theorem}

\newtheorem{lemma}{Lemma}

\newtheorem{assumption}{Assumption}

\newtheorem{example}{Example}

\numberwithin{theorem}{section}
\numberwithin{lemma}{section}
\numberwithin{proposition}{section}

\input def,command

\definecolor{darkgreen}{rgb}{0.0, 0.4, 0.13}

\newcommand{\EE} {\mathbb{E}}

\newcommand{\Xs} {\mathcal{X}}
\newcommand{\Ys} {\mathcal{Y}}

\newcommand{\tF} {\tilde{\bmath{F}}}

\newcommand{\zero} {\bmath{0}}

\newcommand{\F} {\bmath{F}}

\renewcommand{\H} {\bmath{H}}
\newcommand{\I} {\bmath{I}}

\newcommand{\Z} {\bmath{Z}}
\renewcommand{\S} {\bmath{S}}

\renewcommand{\P} {\bmath{P}}

\newcommand{\R} {\bmath{R}}

\renewcommand{\aa} {\bmath{a}}

\renewcommand{\u} {\bmath{u}}

\newcommand{\x} {\bmath{x}}
\newcommand{\y} {\bmath{y}}

\newcommand{\w} {\bmath{w}}
\newcommand{\z} {\bmath{z}}

\title{Fast Extra Gradient Methods for Smooth Structured Nonconvex-Nonconcave Minimax Problems}

\author{%
	Sucheol Lee\\
	Department of Mathematical Sciences\\
	KAIST\\
	Daejeon, Republic of Korea\\
	\texttt{csfh1379@kaist.ac.kr}\\
	\And
	Donghwan Kim\\
	Department of Mathematical Sciences\\
	KAIST\\
	Daejeon, Republic of Korea\\
	\texttt{donghwankim@kaist.ac.kr}\\
}

\begin{document}

\maketitle

\begin{abstract}
Modern minimax problems, such as generative adversarial network and adversarial training, are often under a nonconvex-nonconcave setting, and developing 
an efficient
method 
for such setting is of interest. Recently, two variants of the extragradient (EG) method are studied in that direction. First, a two-time-scale variant of the EG, named EG+, was proposed under a smooth structured nonconvex-nonconcave setting, with a slow $\mathcal{O}(1/k)$ rate on the squared gradient norm, where $k$ denotes the number of iterations. Second, another
variant of EG with an anchoring technique, named 
extra anchored gradient (EAG), was studied under a smooth convex-concave setting, yielding a fast $\mathcal{O}(1/k^2)$ rate on the squared gradient norm. Built upon EG+ and EAG, this paper proposes a two-time-scale EG with anchoring, named fast extragradient (FEG), that has a fast $\mathcal{O}(1/k^2)$ rate on the squared gradient norm 
for smooth structured nonconvex-nonconcave problems;
the corresponding saddle-gradient operator
satisfies the negative comonotonicity condition.
This paper further develops its backtracking line-search version, named FEG-A, for the case where the problem parameters are not available. The stochastic analysis of FEG is also provided.
\end{abstract}

\section{Introduction}
Recently, 
nonconvex-nonconcave minimax problems have received an increased attention in the optimization community and the machine learning community due to their applications to generative adversarial network \cite{goodfellow:14:gan} and adversarial training \cite{madry:18:tdl}.
In this paper, we consider a smooth structured nonconvex-nonconcave minimax problem: 
\begin{align}\label{eq:minimax_problem_i}
\min_{\x\in\reals^{d_x}}\max_{\y\in\reals^{d_y}} f(\x,\y),
\end{align}
where $f:\reals^{d_x}\times\reals^{d_y}\rightarrow \reals$ is smooth and is possibly nonconvex in $\x$ for fixed $\y$, and possibly nonconcave in $\y$ for fixed $\x$;
the saddle-gradient operator
$\F := (\nabla_x f, -\nabla_y f)$
satisfies the negative comonotonicity~\cite{bauschke:21:gmo}.
We construct
an efficient (first-order) method, using a saddle gradient
operator 
$\F$ 
for finding a first-order stationary point 
of the 
problem \eqref{eq:minimax_problem_i}.

So far
little is known under the nonconvex-nonconcave setting, compared to the convex-concave setting.
Recent works \cite{dang:15:otc,diakonikolas:21:emf,liu2020towards,malitsky:20:gro,mertikopoulos:19:omd,song:20:ode,zhou:17:smd} studied extragradient-type methods \cite{korpelevich:76:aem, popov:80:amo} for minimax problems under various structured nonconvex-nonconcave settings.
In other words, they consider 
various non-monotone conditions on $\F$, such as the Minty variational inequality (MVI) condition~\cite{dang:15:otc}, the weak MVI condition~\cite{diakonikolas:21:emf}, and the negative comonotonicity~\cite{bauschke:21:gmo}.\footnote{
Relations between the conditions on $\F$
considered in this paper is summarized in Figure~\ref{fig:relations}.} 
Among them, this paper 
focuses on
the negative comonotonicity condition 
for a Lipschitz continuous $\F$. To the best of our knowledge, the following two-time-scale variant of the extragradient method, named EG+:
\begin{equation}\label{eq:egp}\tag{EG+}
\begin{aligned}
\z_{k+1/2} &= \z_k - \frac{\alpha_k}{\beta}\F\z_k,\\
\z_{k+1} &= \z_k - \alpha_k \F\z_{k+1/2},
\end{aligned}
\end{equation}
is the only known (explicit)\footnote{
A proximal point method
converges under the negative comonotonicity
\cite{bauschke:21:gmo,kohlenbach:21:otp},
but such implicit method is not preferable over explicit methods
in practice due to its implicit nature.
} 
method, using $\F$, that converges under the considered setting\footnote{ 
The EG+ was originally shown to 
work
under
the weak MVI condition of $\F$, which is weaker than the negative comonotonicity.}
\cite{diakonikolas:21:emf}, where $\z_k := (\x_k,\y_k)$.
The EG+, however,
has a slow $\mathcal{O}(1/k)$ rate
on the squared gradient norm.
Note that
a similar two-time-scale approach
has been found to stabilize
the stochastic extragradient method
with unbounded noise variance
\cite{hsieh:20:eau}.



Meanwhile, under the smooth convex-concave setting, recent works \cite{diakonikolas:20:hif,kim2021accelerated,lieder:21:otc,ryu:19:oao,yoon:21:aaf} suggest that Halpern-type \cite{halpern:67:fpo} (or anchoring) methods, 
performing
a convex combination of an initial point $\z_0$ and the last updated point $\z_k$
at each iteration, has 
a fast
$\mathcal{O}(1/k^2)$ rate
in terms of the squared gradient norm.
%
\comment{
\cite{kim2021accelerated,lieder:21:otc} showed that the
(implicit) Halpern iteration~\cite{halpern:67:fpo}
with appropriately chosen step coefficients
has an $\mathcal{O}(1/k^2)$ rate on the squared norm
of a monotone $\F$.
Then,
for a cocoercive $\F$,
an (explicit) version of the Halpern iteration
was studied in~\cite{diakonikolas:20:hif,kim2021accelerated}
that has the same fast rate. 
In addition, \cite{diakonikolas:20:hif} constructed 
a double-loop version of the Halpern iteration
for a Lipschitz continuous and monotone $\F$,
which has a rate $\tilde{\mathcal{O}}(1/k^2)$ on the squared gradient norm,
slower than the rate $\mathcal{O}(1/k^2)$.
While this is promising compared to 
the $\mathcal{O}(1/k)$ rate of
the extragradient methods 
on the squared gradient norm
\cite{ryu:19:oao,solodov:99:aha,yoon:21:aaf},
the computational complexity due to its double-loop nature 
and a relatively slow rate remained a problem.}
%
%
In particular,
\cite{yoon:21:aaf} 
developed the following anchoring variant of the extragradient method, named extra anchored gradient (EAG):
\begin{equation}\label{eq:eag}\tag{EAG}
\begin{aligned}
\z_{k+1/2} &= \z_k + \beta_k(\z_0-\z_k) - \alpha_k \F\z_k,\\
\z_{k+1} &= \z_k + \beta_k(\z_0 - \z_k) - \alpha_k \F\z_{k+1/2}.
\end{aligned}
\end{equation}
This
is the first (explicit) method with a fast $\mathcal{O}(1/k^2)$ rate 
on the squared gradient norm,
when $\F$ satisfies both the Lipschitz continuity and the monotonicity.
\cite{yoon:21:aaf} also showed that such $\mathcal{O}(1/k^2)$ rate
is optimal for first-order methods using a Lipschitz continuous and monotone $\F$. 

%


Built upon both EG+ and EAG,
this paper studies the following class of
two-time-scale anchored extragradient methods, named fast extragradient (FEG):
%
\begin{equation}\tag{Class FEG}\label{alg:special}
\begin{aligned}
\z_{k+1/2} &= \z_k + \beta_k(\z_0-\z_k) - (1-\beta_k)(\alpha_k+2\rho_k)\F\z_k, \\
\z_{k+1} &= \z_k + \beta_k(\z_0-\z_k) - \alpha_k\F\z_{k+1/2} - (1-\beta_k)2\rho_k\F\z_k.
\end{aligned}
\end{equation}
Note that~\eqref{alg:special}
reuses the $\F\z_k$ term in the $\z_{k+1}$ update,
unlike the standard extragradient-type methods,
which we found essential 
for handling the negative comonotonicity condition.
We leave further understanding the use of $\F\z_k$
and the formulation of~\eqref{alg:special}
as future work.
The proposed FEG method (with appropriately chosen step coefficients 
$\alpha_k$, $\beta_k$ and $\rho_k$ discussed later)
has an $\mathcal{O}(1/k^2)$ 
rate on the squared gradient norm, 
under the Lipschitz continuity 
and the negative comonotonicity conditions on $\F$. 
To the best of our knowledge, this is the first accelerated method under the nonconvex-nonconcave setting.
The FEG also has value 
under the smooth convex-concave setting. First, when $\F$ is Lipschitz continuous and monotone, the rate bound of FEG is about 27/4 times smaller than that of EAG. Also note that the rate bound of FEG is only about four times larger than the $\mathcal{O}(1/k^2)$ lower complexity bound of first-order methods under such setting \cite{yoon:21:aaf},
further closing the gap between the lower and upper complexity bounds. 
Second, when $\F$ is cocoercive, 
FEG has a rate faster than that of a version 
of Halpern iteration~\cite{halpern:67:fpo} in \cite{diakonikolas:20:hif}.

We also develop an adaptive variant of FEG, named FEG-A, which updates its parameters, $\alpha_k$ and $\rho_k$ in \eqref{alg:special}, adaptively using a backtracking line-search \cite{beck:09:afi,malitsky2018first,mukkamala2020convex}. FEG requires the knowledge of the two problem parameters for the Lipschitz continuity and the comonotonicity of $\F$. However,
those global parameters can be conservative, and in practice, they are even usually unknown.
For such cases, the FEG-A
adaptively and locally estimates the problem parameters,
while preserving the fast
rate $\mathcal{O}(1/k^2)$ 
on the squared gradient norm
for smooth structured nonconvex-nonconcave minimax problems. 

Lastly, we study a stochastic version of FEG, named S-FEG,
which uses an unbiased stochastic estimate of $\F\z$,
\ie, $\tilde{\F}{\z} = \F\z + \xi$,
instead of $\F\z$ in FEG,
where $\xi$ denotes a stochastic noise.
For a Lipschitz continuous and monotone $\F$,
we provide a convergence analysis in terms of the expected squared gradient norm.
In specific,
we show that the S-FEG is stable with a rate $\mathcal{O}(1/k^2) + \mathcal{O}(\epsilon)$,
when 
the noise variance decreases
in the order of $\mathcal{O}(\epsilon/k)$,
while being unstable otherwise due to error accumulation.
This is similar to the convergence behavior of a stochastic version
of Nesterov's fast gradient method~\cite{nesterov:83:amf,nesterov:05:smo}, 
observed in~\cite{devolder:11:sfo}, for smooth convex minimization.

Our main contributions are summarized as follows.
\begin{itemize}
	\item We propose the FEG method that has an accelerated convergence rate $\mathcal{O}(1/k^2)$ 
	on the squared gradient norm
	for smooth structured nonconvex-nonconcave minimax problems.
	\item We present that the FEG method has a rate faster than 
	those of the EAG and the Halpern iteration for smooth convex-concave problems.
	\item We construct a backtracking line-search version of FEG, named FEG-A, for the case where the Lipschitz constant and comonotonicity parameters of $\F$ are unavailable.
	\item We analyze a stochastic version of FEG, named S-FEG,
	for smooth convex-concave problems.
\end{itemize}

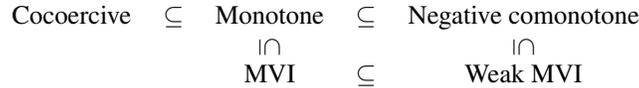
\begin{figure}[b]
    \centering
    \comment{
    \begin{tikzpicture}[scale=0.8, every node/.style={scale=0.8}]
    \draw 
    (-3,-3) rectangle (5,3) (1,3) node [text=black, above] {Weak MVI}
    (0,-0.5) circle (2) (-1.5,1.1) node [text=black, above] {MVI}
    (1.7,1.4) node [text=black, above] {Negative comonotone}
    (0.3,0.7) node [text=black, above] {Monotone}
    (0.5,-0.5) circle (1) (0.5,-0.7) node [text=black, above] {Cocoercive}
    (1.5,0) circle (2.5) (2.5,2.5) node [text=black, above] {Comonotone};
    \end{tikzpicture}
    }
    \begin{tabular}{ccccc}
		Cocoercive & $\subseteq$ & Monotone & $\subseteq$ & Negative comonotone\\
		& &\rotatebox[origin=c]{-90}{$\subseteq$} & & \rotatebox[origin=c]{-90}{$\subseteq$}\\
		& & MVI & $\subseteq$ & Weak MVI
	\end{tabular}
    \caption{Relations between the conditions on $\F$.}
    \label{fig:relations}
\end{figure}

\section{Related work}
\comment{
We are interested in the following minimax problem:
\begin{align}\label{eq:minimax_problem}
\min_{\x\in\reals^{d_x}}\max_{\y\in\reals^{d_y}} f(\x,\y)
\end{align}
where $f:\reals^{d_x}\times\reals^{d_y}\rightarrow \reals$. We say $f$ in \eqref{eq:minimax_problem} is convex-concave (resp. nonconvex-nonconcave) when $f$ is convex (resp. possibly nonconvex) in $\x\in\reals^{d_x}$ for all fixed $\y\in\reals^{d_y}$ and concave (resp. possibly nonconcave) in $\y\in\reals^{d_y}$ for all fixed $\x\in\reals^{d_x}$. For a continuously differentiable function $f$, let us define the \textit{saddle  gradient operator} of $f$, $\F:=(\nabla_{\x}f,-\nabla_{\y}f)$, for later use.
Let $\|\cdot\|$ be the Euclidean norm and $d:=d_x+d_y$.
}
\subsection{Methods for convex-concave minimax problems}
The extragradient method \cite{korpelevich:76:aem}
is one of the widely used methods for solving smooth convex-concave minimax problems
(see, \eg, \cite{dang:15:otc,diakonikolas:21:emf,liu2020towards,malitsky:20:gro,mertikopoulos:19:omd,song:20:ode,zhou:17:smd} for its extensions and applications).
In terms of the \textit{duality gap},  $\max_{\y'\in\mathcal{Y}}f(\x,\y') - \min_{\x'\in\mathcal{X}}f(\x',\y)$, where $\mathcal{X}$ and $\mathcal{Y}$ are compact\footnote{ 
The convergence analysis on the duality gap of the extragradient type methods are generalized under
the unbounded domain assumption in \cite{mokhtari:20:aua,monteiro:10:otc,monteiro2011complexity}.}
domains,
the ergodic iterate of the extragradient-type methods 
\cite{nemirovski:04:pmw,nesterov:07:dea} have 
an $\mathcal{O}(1/k)$ rate.
Such $\mathcal{O}(1/k)$ rate on the duality gap is order-optimal for the first-order methods \cite{nemirovski:83,ouyang:21:lcb},
leaving no room for improvement.
On the other hand, the last iterate of
the extragradient method 
has a slower $\mathcal{O}(1/\sqrt{k})$ rate
on the duality gap, 
under an additional assumption that $\F$ has a 
Lipschitz derivative \cite{golowich:20:lii}. 
In terms of
the \textit{squared gradient norm}, $\|\F\z\|^2$, 
the best iterate of
the extragradient-type methods 
\cite{korpelevich:76:aem,popov:80:amo} have an
$\mathcal{O}(1/k)$ rate
\cite{ryu:19:oao,solodov:99:aha,yoon:21:aaf}.
The last iterate of
the extragradient method also has a rate $\mathcal{O}(1/k)$, 
when $\F$ is further assumed to have a Lipschitz derivative
\cite{golowich:20:lii}.
Unlike the duality gap, the $\mathcal{O}(1/k)$ rate on the squared gradient norm is not optimal \cite{yoon:21:aaf}.
From now on throughout this paper, 
we mainly study and compare the convergence rates on
the squared gradient norm,
which still has room for improvement in convex-concave problems, 
and has meaning for nonconvex-nonconcave minimax problems,
unlike the duality gap.


Recently, \cite{diakonikolas:20:hif,kim2021accelerated,lieder:21:otc,ryu:19:oao,yoon:21:aaf} found that Halpern-type \cite{halpern:67:fpo} (or anchoring) methods 
yield a fast $\mathcal{O}(1/k^2)$ rate
in terms of the squared gradient norm for minimax problems.
\cite{kim2021accelerated,lieder:21:otc} showed that the
(implicit) Halpern iteration~\cite{halpern:67:fpo}
with appropriately chosen step coefficients
has an $\mathcal{O}(1/k^2)$ rate on the squared norm
of a monotone $\F$.
Then,
for a cocoercive $\F$,
an (explicit) version of the Halpern iteration
was studied in~\cite{diakonikolas:20:hif,kim2021accelerated}
that has the same fast rate. 
In addition, \cite{diakonikolas:20:hif} constructed 
a double-loop version of the Halpern iteration
for a Lipschitz continuous and monotone $\F$,
which has a rate $\tilde{\mathcal{O}}(1/k^2)$ on the squared gradient norm,
slower than the rate $\mathcal{O}(1/k^2)$.
While this is promising compared to 
the $\mathcal{O}(1/k)$ rate of
the extragradient methods 
on the squared gradient norm
\cite{ryu:19:oao,solodov:99:aha,yoon:21:aaf},
the computational complexity due to its double-loop nature 
and a relatively slow rate remained a problem.
Very recently,
\cite{yoon:21:aaf} proposed the extra anchored gradient (EAG) method,
which is the first 
(explicit) method
with a fast $\mathcal{O}(1/k^2)$ rate
for smooth convex-concave minimax problems, \ie, for Lipschitz continuous and monotone operators. In addition, \cite{yoon:21:aaf} proved that the EAG is order-optimal by showing that the lower complexity bound of first-order methods is $\Omega(1/k^2)$.

\comment{
Recently, \cite{lieder:21:otc} revealed that the Halpern iteration \cite{halpern:67:fpo} 
(with appropriately chosen step coefficients)
has a fast convergence rate for the fixed point problems with nonexpansive operators. 
Follow-up works \cite{diakonikolas:20:hif, kim2021accelerated, ryu:19:oao, yoon:21:aaf} then showed that the Halpern-type (or anchoring) technique also provides a fast convergence rate $\mathcal{O}(1/k^2)$ on the squared gradient norm for convex-concave minimax problems. 
\cite{kim2021accelerated,lieder:21:otc} 
present that the (implicit) Halpern iteration 
has a fast $\mathcal{O}(1/k^2)$ rate 
for monotone inclusion problems, 
including (possibly non-smooth) convex-concave minimax problems.
\cite{diakonikolas:20:hif} proposed a variant of the Halpern iteration, which has a rate $\mathcal{O}(1/k^2)$ for cocoercive operators \cite{bauschke:21:gmo}.
Also, \cite{diakonikolas:20:hif} proposed a double-loop version of the Halpern iteration, which
has a rate $\tilde{\mathcal{O}}(1/k^2)$ for the case $\F$ is Lipschitz continuous and monotone. 
Very recently,
\cite{yoon:21:aaf} proposed the extra anchored gradient (EAG) method,
which is the first 
accelerated explicit method
for the smooth convex-concave minimax problems, \ie, $\F$ is Lipschitz continuous and monotone. In addition, \cite{yoon:21:aaf} proved that the EAG is order-optimal by showing that the lower complexity bound of first-order methods is $\mathcal{O}(1/k^2)$.
}

\comment{
Even though the duality gap is a widely used suboptimality measure for the minimax problems, it has several disadvantages \cite{diakonikolas:20:hif,yoon:21:aaf}.
First, the duality gap usually cannot be computed directly in practice.
Second, the domains $\mathcal{X}$ and $\mathcal{Y}$ need to be bounded to guarantee that the duality gap has finite value.
Last, the duality gap is not meaningful for the nonconvex-nonconcave minimax problems. Hence, we consider the squared gradient norm, $\|\F\z\|^2$, 
as an alternative suboptimality measure. 
}

\subsection{Methods for nonconvex-nonconcave minimax problems}
Some recent literature considered relaxing the monotonicity condition of the saddle gradient operator to tackle modern nonconvex-nonconcave minimax problems.
For example, 
the Minty variational inequality (MVI) condition,
\ie, there exists $\z_*\in\Z_*(\F)$ satisfying $\inprod{\F\z}{\z-\z_*}\ge 0$ for all $\z\in\reals^d$ where   $\Z_*(\F):=\{\z_*\in\reals^d:\F\z_*=\zero\}$,
is studied in \cite{dang:15:otc,liu:21:foc,liu2020towards,malitsky:20:gro}.
This condition is also studied under the name, the coherence, in \cite{mertikopoulos:19:omd,song:20:ode,zhou:17:smd}.
Moreover, 
\cite{diakonikolas:21:emf} 
considered
a 
weaker condition, named the weak MVI condition, \ie, for some $\rho<0$, there exists $\z_*\in \Z_*(\F)$ satisfying
$\inprod{\F\z}{\z-\z_*}\ge \rho\|\F\z\|^2$ for all $\z\in\reals^d$.
The weak MVI condition is
implied by the negative comonotonicity \cite{bauschke:21:gmo} or, equivalently, the (positive) cohypomonotonicity \cite{combettes:04:pmf}. The comonotonicity will be further
discussed in the upcoming section.

For $L$-Lipschitz continuous 
$\F$, 
\cite{dang:15:otc,song:20:ode} showed that the extragradient-type methods
have an $\mathcal{O}(1/k)$ rate on the squared gradient norm 
under the MVI condition, 
and \cite{diakonikolas:21:emf} 
developed the~\eqref{eq:egp} method
under the weak MVI condition (and thus under the negative comonotonicty),
which also has an $\mathcal{O}(1/k)$ rate on the squared gradient norm. 
To the best of our knowledge, 
there is no known accelerated method 
for the nonconvex-nonconcave setting; 
our proposed FEG method is
the first method to have a fast $\mathcal{O}(1/k^2)$ rate
under the nonconvex-nonconcave setting.
The convergence rates of the 
existing
methods and the FEG on the squared gradient norm 
are summarized in Table~\ref{tab:comparison}.

\begingroup
\begin{table}[h]
    \renewcommand{\arraystretch}{0.7}
    \centering
    \setlength{\tabcolsep}{2pt}
    \caption{Comparison of the convergence rates 
        of the existing extragradient-type methods and the FEG,
        with respect to 
        the squared gradient norm,
        for smooth structured minimax problems,
        under various assumptions on
        the Lipschitz continuous saddle gradient operator $\F$.}
    \label{tab:comparison}
    \begin{tabular}{c c c c c c c}
        \toprule[1pt]
        \multicolumn{2}{c}{\multirow{2}{*}{Method}} & \multicolumn{2}{c}{Convex-concave} & \multicolumn{3}{c}{Nonconvex-nonconcave}\\
        \cmidrule(lr){3-4} \cmidrule(lr){5-7}
         & & Cocoercive & Monotone & Negative comonotone & MVI & Weak MVI\\
        \midrule[1pt]
        \multirow{2}{*}{Normal}& EG \cite{dang:15:otc,song:20:ode} & $\mathcal{O}(1/k)$ & $\mathcal{O}(1/k)$ &  & $\mathcal{O}(1/k)$ &\\
        & EG+ \cite{diakonikolas:21:emf} & $\mathcal{O}(1/k)$ & $\mathcal{O}(1/k)$ & $\mathcal{O}(1/k)$ & $\mathcal{O}(1/k)$ & $\mathcal{O}(1/k)$\\
        \midrule
        \multirow{3}{*}{Accelerated} 
        & Halpern \cite{halpern:67:fpo,diakonikolas:20:hif} & $\mathcal{O}(1/k^2)$ & $\tilde{\mathcal{O}}(1/k^2)$ & & & \\
        & EAG \cite{yoon:21:aaf} & $\mathcal{O}(1/k^2)$ & $\mathcal{O}(1/k^2)$ & & & \\
        & FEG (this paper) & $\mathcal{O}(1/k^2)$ & $\mathcal{O}(1/k^2)$ & $\mathcal{O}(1/k^2)$ & & \\
        \bottomrule[1pt]
    \end{tabular}
\end{table}
\endgroup

\section{Preliminaries}\label{sec:preliminaries}

The followings are the two main assumptions for the saddle gradient operator $\F$
of the smooth structured nonconvex-nonconcave problem~\eqref{eq:minimax_problem_i}.
Under such assumptions, we 
develop efficient methods that
find a first-order stationary point $\z_*\in\Z_*(\F)$ where $\Z_*(\F):=\{\z_*\in\reals^d:\F\z_*=\zero\}$.
\begin{assumption}[$L$-Lipschitz continuity]\label{assum:lipschitz}
	For some $L\in (0,\infty)$, $\F$ satisfies
	\begin{align*}
	\|\F\z-\F\z'\|\le L\|\z-\z'\|,\quad \forall \z,\z'\in\reals^d.
	\end{align*}
\end{assumption}
\begin{assumption}[$\rho$-Comonotonicity]\label{assum:comonotone}
	For some $\rho\in\big(-\frac{1}{2L},\infty\big)$, 
	$\F$ satisfies
	\begin{align*}
	\inprod{\F\z-\F\z'}{\z-\z'}\ge \rho\|\F\z-\F\z'\|^2, \quad \forall \z,\z'\in\reals^d.
	\end{align*}
\end{assumption}

The $\rho$-comonotonicity consists of three cases 
depending on the choice of $\rho$;  
the negative comonotonicity when $\rho<0$, 
the monotonicity when $\rho=0$, 
and the cocoercivity when $\rho>0$. 
The negative comonotonicity is weaker than 
the other two,
and is the main focus of this paper.
The following is an examplary nonconvex-nonconcave condition
that is stronger than the negative comonotonicity
\cite{bauschke:21:gmo,combettes:04:pmf}.
\comment{
First,
the case $\rho$<0 corresponds to the negative comonotonicity \cite{bauschke:21:gmo} and the |$\rho$|-cohypomonotonicity \cite{combettes:04:pmf}, which is weaker than the monotonicity and is the main focus of this paper.
Second, the case $\rho$=0 corresponds to the monotoncity.
Finally, for the case $\rho>0$, the $\rho$-comonotonicity is stronger than the monotonicity and studied under the name, $\rho$-cocoercivity. As a remark, if $\F$ is $\rho$-cocoercive,
then it is also an $\frac{1}{\rho}$-Lipschitz continuous operator.
Hence, when we consider an $L$-Lipschitz continuous and $\rho$-comonotone operator with $\rho>0$, we may assume that $\rho\in(0,\frac{1}{L}]$ without loss of generality.
}
\begin{example}
	\label{ex:comonotone,relation}
	Let $f$ be twice continuously differentiable and 
	$\gamma$-weakly-convex-weakly-concave.
	Further assume that $f$ satisfies
	\begin{align}
	\nabla_{\x\x}^2f + \nabla_{\x\y}^2f(\eta\I - \nabla_{\y\y}^2f)^{-1}\nabla_{\y\x}^2f &\succeq \alpha\I, \label{eq:interaction} \\
	-\nabla_{\y\y}^2f + \nabla_{\y\x}^2f(\eta\I + \nabla_{\x\x}^2f)^{-1}\nabla_{\x\y}^2f &\succeq \alpha\I, \nonumber
	\end{align}
	for some $\alpha\ge0$ and $\eta>\gamma$,
	named $\alpha\ge0$-interaction dominant
	condition in~\cite{grimmer:20:tlo}.
	Then, the saddle gradient of $f$ 
	satisfies the $-\frac{1}{\eta}$-negative comonotonicity.
	(See Appendix~\ref{appx:comonotone,relation}.)
	For any $\gamma$-weakly-convex-weakly-concave function,
	the condition~\eqref{eq:interaction} holds with $\alpha=-\gamma<0$.
	Its extreme case is $f(x,y)=-\frac{\gamma}{2}x^2 + \frac{\gamma}{2}y^2$,
	where there is no interaction between $x$ and $y$.
	On the other hand, when the the second terms in the left-hand side of~\eqref{eq:interaction}
	are sufficently positive definite,
	a nonconvex-nonconave function satisfies
	the condition~\eqref{eq:interaction} with a nonnegative $\alpha$.
	In specific, the $\alpha\ge0$-interaction dominant condition 
	is satisfied 
	when the interaction term of Hessian $\nabla_{\x\y}^2f$
	is dominating any negative curvature in 
	Hessians $\nabla_{\x\x}^2f$ and $-\nabla_{\y\y}^2f$
	\cite{grimmer:20:tlo}.
\end{example}
We next present our proposed FEG, and
illustrate that the FEG outperforms existing methods such as EG+, EAG, and the Halpern iteration, for each three comonoticity case, respectively.

\comment{
Recently, \cite{diakonikolas:21:emf} studied the weak Minty variational inequality (MVI) condition, there exists $\z_*\in\Z_*(\F)$ satisfying
\begin{align*}
\inprod{\F\z}{\z-\z_*}\ge \rho\|\F\z\|^2 \qquad\forall \z\in\reals^d
\end{align*}
for some $\rho<0$, which is weaker than the negative comonotonicity. 
}

\section{Fast extragradient (FEG) method 
for Lipschitz continuous and comonotone operators} \label{sec:FEG}
\comment{
More specifically, we consider the method in 
the following form, named two-time-scale and anchored extragradient (TAEG) method.
\begin{equation}\tag{TAEG}\label{alg:special}
\begin{aligned}
\z_{k+1/2} &= \z_k + \beta_k(\z_0-\z_k) - (1-\beta_k)(\alpha_k+2\rho_k)\F\z_k\\
\z_{k+1} &= \z_k + \beta_k(\z_0-\z_k) - \alpha_k\F\z_{k+1/2} - (1-\beta_k)2\rho_k\F\z_k
\end{aligned}
\end{equation}
for all $k\ge 0$.
}

This section considers an instance
of \eqref{alg:special} with $\alpha_k = \frac{1}{L}$, $\beta_k = \frac{1}{k+1}$, and $\rho_k = \rho$ for all $k\ge 0$.
The resulting method, named FEG, is illustrated in Algorithm~\ref{alg:feg},
which has an
$\mathcal{O}(1/k^2)$ fast rate
with respect to the squared gradient norm,
in Theorem~\ref{thm:convergence_deterministic}.
The proof of Theorem~\ref{thm:convergence_deterministic}
is provided in Section~\ref{sec:conv_analysis}.
\begin{algorithm}[h]
	\caption{Fast extragradient (FEG) method}
	\label{alg:feg}
	\begin{algorithmic}
		\State {\bf Input:} $\z_0\in\reals^d$, $L\in(0,\infty)$, $\rho\in\big(-\frac{1}{2L},\infty\big)$
		\For{$k=0,1,\ldots$}
		\vspace{-15px}
		\State
		\begin{align*}
		\z_{k+1/2} &= \z_k + \frac{1}{k+1}(\z_0-\z_k) - \Big(1-\frac{1}{k+1}\Big)\Big(\frac{1}{L}+2\rho\Big)\F\z_k \\ 
        \z_{k+1} &= \z_k + \frac{1}{k+1}(\z_0-\z_k) - \frac{1}{L}\F\z_{k+1/2} - \Big(1-\frac{1}{k+1}\Big)2\rho\F\z_k.
		\end{align*}
		\EndFor
	\end{algorithmic}
\end{algorithm}
\comment{
\textbf{Fast extragradient (FEG)}
\begin{equation*}
\begin{aligned}
\z_{k+1/2} &= \z_k + \frac{1}{k+1}(\z_0-\z_k) - \Big(1-\frac{1}{k+1}\Big)\Big(\frac{1}{L}+2\rho\Big)\F\z_k\\
\z_{k+1} &= \z_k + \frac{1}{k+1}(\z_0-\z_k) - \frac{1}{L}\F\z_{k+1/2} - \Big(1-\frac{1}{k+1}\Big)2\rho\F\z_k.
\end{aligned}
\end{equation*}
}

\begin{theorem}\label{thm:convergence_deterministic}
	For the $L$-Lipschitz continuous and $\rho$-comonotone operator $\F$ with $\rho>-\frac{1}{2L}$ and for any $\z_*\in\Z_*(\F)$, the sequence $\{\z_k\}_{k\ge 0}$ generated by FEG satisfies, for all $k\ge 1$,
	\begin{align}
	\|\F\z_k\|^2\le \frac{4\|\z_0-\z_*\|^2}{\Big(\frac{1}{L}+2\rho\Big)^2k^2}
	\label{eq:feg}
	.
	\end{align}
\end{theorem}
The following example shows that 
the bound~\eqref{eq:feg} of the FEG
is exact for $\rho=0$ and $k=4l+2$.
The bound~\eqref{eq:feg} is not known to be exact in general,
and we leave finding the exact bound as future work.
\begin{example} \label{lem:worst_case}
	Let $f:\reals\times\reals\rightarrow \reals$ be $f(x,y) = Lxy$. 
	Its saddle gradient operator and solution are $\F(x,y) = (Ly,-Lx)$
	and $\z_*=(0,0)$, respectively. 
	For the initial point $\z_0 = (x_0,y_0) = (1,0)$, the sequence $\{\z_k\}_{k\ge0}$ generated by FEG satisfies
	$
	\z_{4l+2} = \big(0,\frac{1}{2l+1}\big)
	$
	for all $l\ge 0$. Hence, $\|\F\z_{4l+2}\|^2 = \frac{L^2}{(2l+1)^2}= \frac{4L^2\|\z_0-\z_*\|^2}{(4l+2)^2}$ for all $l\ge 0$.
	(See Appendix~\ref{appx:worst_case}.)
\end{example}
We next compare 
the rate bound~\eqref{eq:feg}
with existing analyses for the three cases $-\frac{1}{2L}<\rho<0$, $\rho=0$, and $\rho>0$.

\subsection{Comparison to EG+ under the negative comonotonicity ($\rho<0$)}
Under the negative comonotonicity with $-\frac{1}{8L}<\rho<0$, the \eqref{eq:egp} method with $\alpha_k=\frac{1}{2L}$ and $\beta=\frac{1}{2}$ has an $\mathcal{O}(1/k)$ rate on the squared gradient norm.
To the best of our knowledge, this is the best known rate,  
and the FEG has a faster
$\mathcal{O}(1/k^2)$ rate with a wider region of convergence $-\frac{1}{2L}<\rho<0$.
\comment{
EG+ with $\alpha_k = \frac{1}{2L}$ and $\beta = \frac{1}{2}$ in \eqref{eq:egp} has $\mathcal{O}(1/k)$ best-iterate convergence rate,
$\min_{0\le i \le k}\|\F\z_{k+1/2}\|^2\le\frac{2L\|\z_0-\z_*\|^2}{(k+1)(1/(4L)+2\rho)}$ for $\rho\in\big(-\frac{1}{8L},0\big)$.
Since the $\rho$-weak MVI condition is weaker than the $\rho$-comonotone condition for $\rho<0$, this convergence result of EG+ can be directly implemented under our $L$-Lipschitz continuous and $\rho$-comonotone setting where $\rho<0$.
To the best of our knowledge, except FEG, EG+ is the only existing method that finds the zeros of $\F$ satisfying the Lipschitz continuity and negative comonotonicity.
Compared to EG+, FEG has an accelerated $\mathcal{O}(1/k^2)$ last iterate convergence rate and works for a wider region $\rho\in\big(-\frac{1}{2L},0\big)$.
}
\subsection{Comparison to EAG under the monotonicity ($\rho = 0$)}
For an $L$-Lipschitz continuous and monotone operator $\F$, \cite{yoon:21:aaf} proposed two EAG methods, named EAG-C and EAG-V,
with same $\beta_k = \frac{1}{k+2}$
but with different choices of $\alpha_k$. 
EAG-C 
sets $\alpha_k$ to be a constant $\frac{1}{8L}$
for all $k\ge 0$ in \eqref{eq:eag}, 
and has a large constant 
$260$ in its convergence rate, $\|\F\z_k\|^2\le\frac{260 L^2\|\z_0-\z_*\|^2}{(k+1)^2}$ for all $k\ge 0$. 
On the other hand, while EAG-V requires a complicated recursive update for $\{\alpha_k\}$, 
$\alpha_{k+1} = \frac{\alpha_k}{1-\alpha_k^2L^2}\big(1-\frac{(k+2)^2}{(k+1)(k+3)}\alpha_k^2L^2\big)$
for all $k\ge 0$, 
with $\alpha_0=\frac{0.618}{L}$, 
its rate
has a smaller constant $27$. 

\comment{
On the other hand, EAG-V has smaller constant factor in its convergence rate, $\|\F\z_k\|^2\le\frac{27 L^2\|\z_0-\z_*\|^2}{(k+1)(k+2)}$ for all $k\ge 0$, but its step size requires complicated recursive update,
$
\alpha_{k+1} = \frac{\alpha_k}{1-\alpha_k^2L^2}\Big(1-\frac{(k+2)^2}{(k+1)(k+3)}\alpha_k^2L^2\Big)
$
for $k\ge 0$ where $\alpha_0=\frac{0.618}{L}$ in \eqref{eq:eag}.
}

The FEG takes a constant $\alpha_k = \frac{1}{L}$, 
unlike EAG-V,
but has an even smaller constant $4$ 
in its convergence rate $\|\F\z_k\|^2\le \frac{4 L^2\|\z_0-\z_*\|^2}{k^2}$
for $\rho=0$. 
Therefore, the FEG with $\rho=0$ has about $260/4$-times and $27/4$-times faster convergence rate compared to those of EAG-C and EAG-V, respectively.
Furthermore, the rate bound of FEG with $\rho=0$ is only about $4$-times larger
than the lower complexity bound 
of first-order methods under the considered setting
\cite{yoon:21:aaf},
reducing the gap between the lower and upper complexity bounds
from $27$ to $4$.


\subsection{Comparison to the Halpern iteration under the cocoercivity ($\rho>0$)}
For a $\rho$-cocoercive operator $\F$, 
%
an (explicit) version of Halpern iteration \cite{halpern:67:fpo},
studied in \cite{diakonikolas:20:hif},
has a fast rate,
$\|\F\z_k\|^2\le\frac{\|\z_0-\z_*\|^2}{\rho^2k^2}$.
Note that while the $\rho$-cocoercivity 
implies the $\frac{1}{\rho}$-Lipschitz continuity,
there is case where the $\rho$-cocoercive 
(and thus Lipschitz continuous) operator 
has a Lipschitz constant $L$
smaller than $\frac{1}{\rho}$.
Since $L \le \frac{1}{\rho}$,
the FEG has a rate $\|\F\z_k\|^2\le\frac{4\|\z_0-\z_*\|^2}{(1/L + 2\rho)^2k^2} = \frac{4\|\z_0-\z_*\|^2}{9\rho^2k^2}$ 
that is faster than that of Halpern iteration. However, 
if we take into account
that the FEG requires 
computing
the saddle gradient 
twice per iteration, 
unlike Halpern iteration studied in \cite{diakonikolas:20:hif}, 
the FEG method has a slower rate in terms of the number of gradient computations.
If we narrow down to the case $L < \frac{1}{2\rho}$,
the FEG has a faster rate, $\|\F\z_k\|^2\le\frac{4\|\z_0-\z_*\|^2}{(1/L + 2\rho)^2k^2}< \frac{\|\z_0-\z_*\|^2}{4\rho^2k^2}$. 
For such case, 
the FEG has a rate faster than that of the Halpern iteration, 
even in terms of the number of gradient computations.

\subsection{Toy example}
We performed a toy experiment on 
a simple quadratic function,
	$f(x,y) = \frac{\rho L^2}{2}x^2 + L\sqrt{1-\rho^2 L^2}xy - \frac{\rho L^2}{2} y^2,$
which has an $L$-Lipschitz continuous and $\rho$-comonotone saddle gradient. For the case $\rho=-\frac{1}{3L}$ and $L=1$, 
Figure~\ref{fig:comonotone_example} illustrates
that the FEG converges with an accelerated rate
whereas
EG+, EAG-C, EAG-V, and the (explicit) version of Halpern iteration \cite{diakonikolas:20:hif} diverge.
This example presents that the existing guarantees
on convergence and acceleration of the aforementioned methods under the convex-concave setting
do not generalize to the nonconvex-nonconcave setting.


\begin{figure}
	\centering
	\includegraphics[width=.5\textwidth]{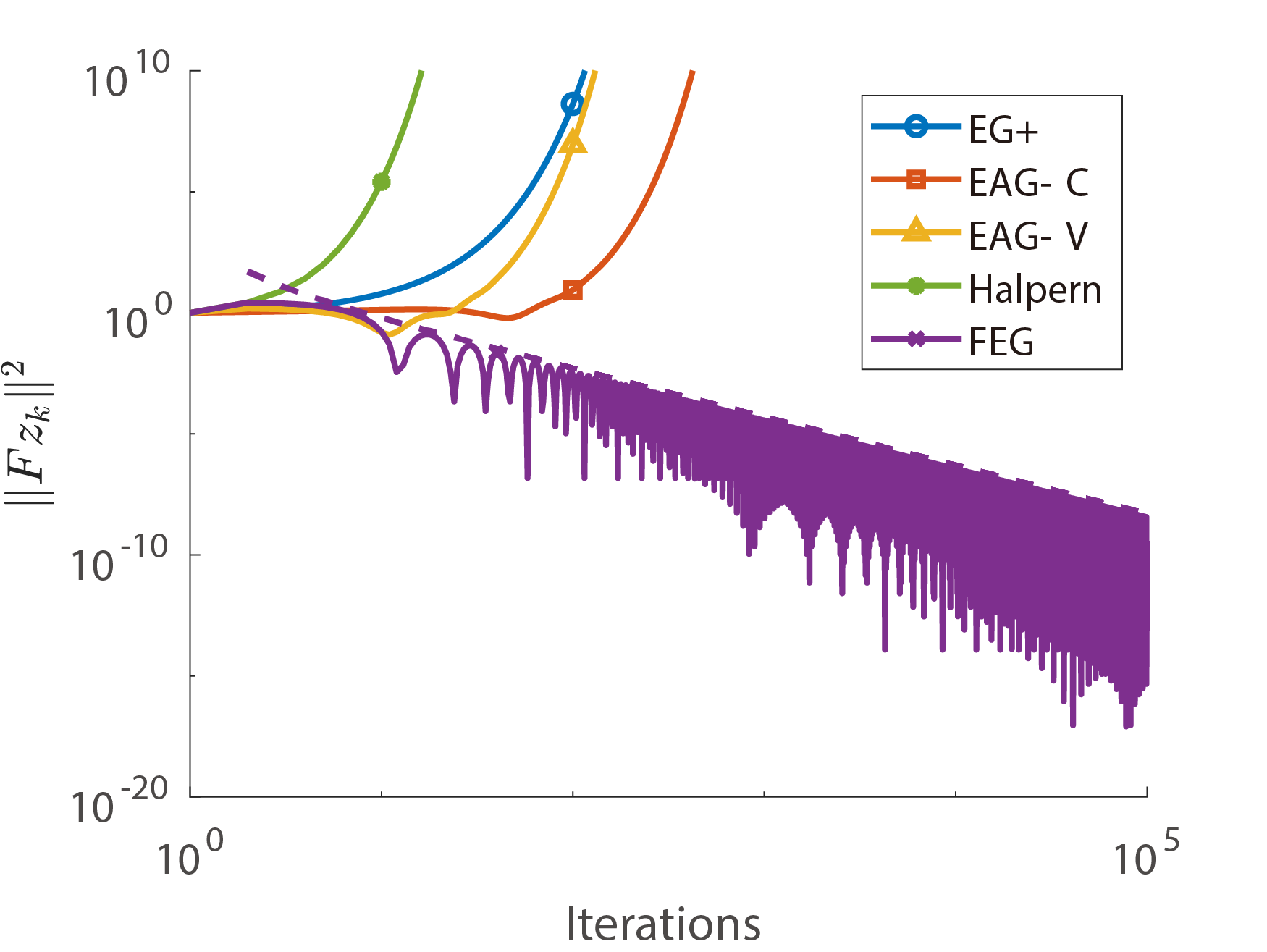}
	\caption{Numerical result with $f(x,y) = -\frac{1}{6}x^2 + \frac{2\sqrt{2}}{3}xy + \frac{1}{6}y^2$. The dashed line represents the theoretical bound~\eqref{eq:feg} of FEG.}
	\label{fig:comonotone_example}
\end{figure}

\section{FEG with backtracking line-search}\label{sec:FEG-A}
The FEG requires the knowledge of the two global parameters $L$ and $\rho$ for Lipschitz continuity and comonotonicity, respectively. Those global parameters are often difficult to compute in practice and can be locally conservative. To handle these two disadvantages, we employ the backtracking line-search technique \cite{beck:09:afi,malitsky2018first,mukkamala2020convex} in FEG. 
We adaptively decrease 
the two step size parameters,
$\tau$ and $\eta$, to satisfy the both conditions, the local $\frac{1}{\tau}$-Lipschitz continuity and the $\frac{\eta-\tau}{2}$-comonotonicity.\footnote{
In specific, $\tau$ and $\eta$ locally estimate
$\frac{1}{L}$ and $\frac{1}{L} + 2\rho$, respectively.
One could have directly estimate $\rho$, instead of $\frac{1}{L} + 2\rho$,
but this complicates the line-search process
to handle both positive and negative values of $\rho$,
unlike our choice of $\eta$ in FEG-A.}
A pseudocode of the resulting method, named FEG-A, is 
illustrated in Algorithm~\ref{alg:feg_a}. For a detailed description of 
the FEG-A, see Algorithm~\ref{alg:feg_a_detailed} in Appendix~\ref{appx:feg_a_detailed}.

\begin{algorithm}[h]
	\caption{Fast extragradient method with adaptive step size (FEG-A)}
	\label{alg:feg_a}
	\begin{algorithmic}
		\State {\bf Input:} $\z_0\in\reals^d$, 
		$\tau_{-1}\in(\max\{0,-2\rho\},\infty)$, $\eta_0\in(0,\infty)$,
		$\delta\in(0,1)$
		\State Find the smallest nonnegative integer $i_0$ such that $\hat{\z} = \z_0 - 
		\tau_{-1}(1-\delta)^{i_0}\F\z_0$ satisfies
		$
		\|\F\hat{\z}-\F\z_0\| \le \frac{1}{\tau_{-1}(1-\delta)^{i_0}}\|\hat{\z}-\z_0\|$. 
		\State $\tau_0 = \tau_{-1}(1-\delta)^{i_0}$, $\z_1 = \z_0 - \tau_0\F\z_0$.
		\For{$k=1,2,\ldots$}
		\State $i_k=j_k=0$.
		\State 
		Increase each $i_k$ and $j_k$ one by one until
		\begin{align*}
		\hat{\z}_{k+1/2} &= \z_k + \frac{1}{k+1}(\z_0-\z_k) - \Big(1-\frac{1}{k+1}\Big)\eta_{k-1}(1-\delta)^{j_k} \F\z_k \qquad\text{and}\\
		\hat{\z}_{k+1} &= \z_k + \frac{1}{k+1}(\z_0-\z_k) - \tau_{k-1}(1-\delta)^{i_k}\F\z_{k+1/2}
		    \\
		&\quad\quad\; - \Big(1-\frac{1}{k+1}\Big)(\eta_{k-1}(1-\delta)^{j_k}-\tau_{k-1}(1-\delta)^{i_k})\F\z_k
		\end{align*}
		\State satisfy both conditions,
		\begin{align*}
		    \|\F\hat{\z}_{k+1}-\F\hat{\z}_{k+1/2}\| &\le \frac{1}{\tau_{k-1}(1-\delta)^{i_k}}\|\hat{\z}_{k+1}-\hat{\z}_{k+1/2}\| \qquad\text{and}\\
		    \inprod{\F\hat{\z}_{k+1}-\F\z_k}{\hat{\z}_{k+1}-\z_k} &\ge \frac{\eta_{k-1}(1-\delta)^{j_k}-\tau_{k-1}(1-\delta)^{i_k}}{2}\|\F\hat{\z}_{k+1}-\F\z_k\|^2.
		\end{align*}
		\State $\tau_k = \tau_{k-1}(1-\delta)^{i_k}$, $\eta_k = \eta_{k-1}(1-\delta)^{j_k}$, $\z_{k+1} = \hat{\z}_{k+1}$.
		\EndFor
	\end{algorithmic}
\end{algorithm}
\comment{
\begin{algorithm}[h]
	\caption{Fast extragradient method with adaptive step size (FEG-A)}
	\label{alg:feg_a}
	\begin{algorithmic}
		\State {\bf Input:} $\z_0$, $\tau_{-1}\in\big(\frac{1-\delta}{L},\infty\big)$, $\eta_{0}\in \big(\frac{(1-\delta)^2}{L}+(1-\delta)2\rho,\infty\big)$, $\delta\in(0,1)$
		\State Find the smallest nonnegative integer $i_0$ such that $\hat{\z} = \z_0 - \hat{\tau}\F\z_0$ satisfies
		$
		\|\F\hat{\z}-\F\z_0\| \le \frac{1}{\tau_{-1}(1-\delta)^{i_0}}\|\hat{\z}-\z_0\|$. 
		\State $\tau_0 = \tau_{-1}(1-\delta)^{i_0}$, $\z_1 = \z_0 - \tau_0\F\z_0$.
		\For{$k=1,2,\ldots$}
		\State $i_k=j_k=0$, $\text{searching}=\text{True}$
		\While{$\text{searching}=\text{True}$}
		\State $\text{searching}=\text{False}$, $\tau_k = \tau_{k-1}(1-\delta)^{i_k}$, $\eta_k = \eta_{k-1}(1-\delta)^{j_k}$
		\begin{gather*}
		\z_{k+1/2} = \z_k + \frac{1}{k+1}(\z_0-\z_k) - \Big(1-\frac{1}{k+1}\Big)\eta_k \F\z_k \quad\text{and}\\
		\z_{k+1} = \z_k + \frac{1}{k+1}(\z_0-\z_k) - \tau_k\F\z_{k+1/2} - \Big(1-\frac{1}{k+1}\Big)(\eta_k-\tau_k)\F\z_k
		\end{gather*}
		\If{$\|\F\z_{k+1}-\F\z_{k+1/2}\| > \frac{1}{\tau_k}\|\z_{k+1}-\z_{k+1/2}\|$}
		\State $i_k \leftarrow i_k+1$
		\State $\text{searching}=\text{True}$
		\EndIf
		\If{$\inprod{\F\z_{k+1}-\F\z_k}{\z_{k+1}-\z_k} < \frac{\eta_k-\tau_k}{2}\|\F\z_{k+1}-\F\z_k\|^2$}
		\State $j_k \leftarrow j_k +1$
		\State $\text{searching}=\text{True}$
		\EndIf
		\EndWhile
		\EndFor
	\end{algorithmic}
\end{algorithm}
}
\comment{
To guarantee that the FEG-A is well-defined, the total numbers of updates for  $\{\tau_k\}_{k\ge -1}$ and $\{\eta_k\}_{k\ge 0}$ in FEG-A, which corresponds to $\sum_{l=0}^\infty i_l$ and $\sum_{l=1}^\infty j_l$, respectively, must be finite numbers. In fact, there exist positive lower bounds for $\{\tau_k\}_{k\ge -1}$ and $\{\eta_k\}_{k\ge 0}$ in FEG-A by the following lemma.
By noting that $\tau_k = \tau_{-1}(1-\delta)^{\sum_{l=0}^k i_l}$ and $\eta_k = \eta_0(1-\delta)^{\sum_{l=1}^k j_l}$, the positive lower bounds for $\{\tau_k\}_{k\ge -1}$ and $\{\eta_k\}_{k\ge 0}$ implies that $\sum_{k=0}^\infty i_k<\infty$ and $\sum_{k=1}^\infty j_k<\infty$.
}
The following lemma shows that each of the 
nonincreasing
sequences $\{\tau_k\}_{k\ge 0}$ and $\{\eta_k\}_{k\ge 0}$ of the FEG-A 
has a positive lower bound, and thus FEG-A is well-defined\footnote{
This requires one to chooses $\tau_{-1}$
strictly greater than the unknown value $-2\rho$ when $\rho<0$.},
under the condition
$\rho>-\frac{\tau_k}{2}$. 
This condition for $\rho$ can be weaker 
than the condition 
$\rho>-\frac{1}{2L}$ of FEG, 
since the local Lipschitz parameter $\frac{1}{\tau_k}$
can be smaller than $L$.
This is another benefit of using a backtracking line-search
in FEG, over the standard FEG.
\begin{lemma}\label{lem:lower_bounds_of_stepsizes}
	For the $L$-Lipschitz and $\rho$-comonotone operator $\F$ 
	and a given constant $\delta\in(0,1)$,
	the step size $\tau_k$ 
	of FEG-A is lower bounded by a positive value $\underline{\tau}:=\min\big\{\tau_{-1},\frac{1-\delta}{L}\big\}$ for all $k\ge 0$,
	and if $\rho > -\frac{\tau_k}{2}$, 
	the step size $\eta_k$
	is lower bounded by a positive value
	$\min\big\{\eta_0, (1-\delta)\big(\tau_k+2\rho\big)\big\}$ for all $k\ge 1$.
\end{lemma}

The FEG-A method also has the following $\mathcal{O}(1/k^2)$ 
rate with respect to the squared gradient norm
in Theorem~\ref{thm:convergence_adaptive}, 
when $\rho>-\frac{\tau_k}{2}$. 
The proof is provided in Section~\ref{sec:conv_analysis} and 
Appendix~\ref{appx:feg_a_thm}.
\begin{theorem}\label{thm:convergence_adaptive}
	For the $L$-Lipschitz and $\rho$-comonotone operator $\F$ 
	and for any $\z_*\in\Z_*(\F)$, the sequence $\{\z_k\}_{k\ge 0}$ generated by FEG-A satisfies 
	\begin{align*}
	\|\F\z_k\|^2 \le \frac{4\|\z_0-\z_*\|^2}{\left((k-1)\eta_k + \tau_k + 2\rho \right)^2} 
	\end{align*}
	for all $k\ge 1$, if $\rho>-\frac{\tau_k}{2}$. 
\end{theorem}
This rate bound of FEG-A reduces to
that of FEG 
in Theorem~\ref{thm:convergence_deterministic},
when we choose $\tau_{-1} = \frac{1}{L}$
and $\eta_0 = \frac{1}{L} + 2\rho$ for FEG-A.

\comment{
\textbf{Remark.} To achieve $\|\F\z_k\|^2\le \epsilon$ using FEG-A, we need
\begin{align}\label{eq:aega_oracle_complexity}
\frac{4\|\z_0-\z_*\|}{\Big(\frac{1-\delta}{L}+2\rho\Big)(1-\delta)\sqrt{\epsilon}}-\frac{2\delta}{1-\delta}+2\log_{1-\delta}\frac{1-\delta}{L\tau_{-1}} + 2\log_{1-\delta}\frac{\frac{(1-\delta)^2}{L}+(1-\delta)2\rho}{\rho_0}
\end{align} number of gradient estimate.
Therefore, taking smaller $\delta$ makes the FEG-A get faster convergence rate and work for a wider region $\rho>-\frac{1-\delta}{2L}$, but it increases the number of gradient estimate for fitting parameters which corresponds to the last two terms of \eqref{eq:aega_oracle_complexity}.
}

\section{FEG under stochastic setting}\label{sec:stochastic}
When exactly computing
$\F\z$ is expensive in practice, 
one usually instead
consider its stochastic estimate for computational efficiency
(see, \eg,
\cite{hsieh:19:otc,juditsky:11:svi,mertikopoulos:19:omd,nemirovski:09:rsa,ryu:19:oao,song:20:ode,zhou:17:smd}). 
This section also considers using a stochastic oracle in FEG for smooth convex-concave problems. 
In specific, this section assumes that we only have access to a noisy saddle gradient oracle, $\tF\z_{k/2} = \F\z_{k/2} + \xi_{k/2}$, where $\{\xi_{k/2}\}_{k\ge 0}$ are independent random variables satisfying $\EE[\xi_{k/2}]=0$ and $\EE[\|\xi_{k/2}\|^2]= \sigma_{k/2}^2$ for all $k\ge 0$.
Under this setting, we study a stochastic first-order method, named stochastic fast extragradient (S-FEG) method, illustrated in Algorithm~\ref{alg:stoc}.



\begin{algorithm}[h]
	\caption{Stochastic fast extragradient (S-FEG) method}
	\label{alg:stoc}
	\begin{algorithmic}
		\State {\bf Input:} $\z_0\in\reals^d$, $L\in(0,\infty)$. 
		\For{$k=0,1,\ldots$}
		\vspace{-15px}
		\State
		\begin{align*}
		\z_{k+1/2} &= \z_k + \frac{1}{k+1}(\z_0-\z_k) - \Big(1-\frac{1}{k+1}\Big)\frac{1}{L}\tF\z_k \\ 
        \z_{k+1} &= \z_k + \frac{1}{k+1}(\z_0-\z_k) - \frac{1}{L}\tF\z_{k+1/2}
		\end{align*}
		\EndFor
	\end{algorithmic}
\end{algorithm}
\comment{
Under this setting, we consider the 
following stochastic FEG (S-FEG):
\begin{equation}\label{alg:stoc}\tag{S-FEG}
\begin{aligned}
\z_{k+1/2} &= \z_k + \frac{1}{k+1}(\z_0-\z_k) - \Big(1-\frac{1}{k+1}\Big)\frac{1}{L}\tF\z_k\\
\z_{k+1} &= \z_k + \frac{1}{k+1}(\z_0-\z_k) - \frac{1}{L}\tF\z_{k+1/2}.
\end{aligned}
\end{equation}
}
\comment{
As in the previous section, our analysis relies on the same potential function, $V_k = a_k\|\F\z_k\|^2 - b_k\inprod{\F\z_k}{\z_0-\z_k}$. The expectation of the potential function is no more nonincreasing, but we can give lower bound on $\EE[V_k]-\EE[V_{k+1}]$ consist of $\sigma_k^2$ and $\sigma_{k+1/2}^2$ by the following lemma.
\begin{lemma}\label{lem:potential_stoc}
	Let $\{\z_k\}_{k\ge 0}$ be the sequence obtained by \eqref{alg:stoc}.
	Assume that $\{\alpha_k\}_{k\ge 0}$ and $\{\beta_k\}_{k\ge 0}$ satisfy $\beta_0 = 1$, $\{\beta_k\}_{k\ge 1}\subseteq(0,1)$, $\alpha_0\in(0,\infty)$, $\alpha_k\in\Big(0,\frac{1}{L_k}\Big]$ for all $k\ge 1$, and
	\begin{align*}
	\frac{(1-\beta_{k+1})\alpha_{k+1}}{2\beta_{k+1}}
	\le \frac{\alpha_k}{2\beta_k}
	\end{align*}
	for all $k\ge 0$. Then 
	$ 
	V_k = a_k\|\F\z_k\|^2-b_k\inprod{\F\z_k}{\z_0-\z_k}
	$ 
	satisfies
	\begin{align*}
	    \EE[V_0]-\EE[V_1] &\ge - \Big(\frac{L^2\alpha_0^3}{2}+L\alpha_0^2\Big)\sigma_0^2 \qquad\text{and}\\
	\EE[V_k]-\EE[V_{k+1}] &\ge -\frac{b_k\alpha_k(1+2L\alpha_k)}{2\beta_k}\Big((1-\beta_k)\sigma_k^2 + \frac{1}{1-\beta_k}\sigma_{k+1/2}^2\Big)
	\qquad\text{for all $k\ge 1$}
	\end{align*}
	where $a_0 = \frac{\alpha_0(L_0^2\alpha_0^2-1)}{2}$, $b_0 = 0$, $b_1 = 1$,
	\begin{align*}
	a_k &= \frac{b_k(1-\beta_k)\alpha_k}{2\beta_k} \quad\text{for all $k\ge 1$, and}\\
	b_{k+1} &= \frac{b_k}{1-\beta_k} \quad\text{for all $k \ge 1$}.
	\end{align*}
\end{lemma}
}

The following theorem  
provides
an upper bound of the expected squared gradient norm for the S-FEG.
(See Appendix~\ref{appx:s_feg_thm} for the proof.)
\begin{theorem}\label{thm:convergence_stoc}
	Let $\tF\z_{k/2} = \F\z_{k/2} + \xi_{k/2}$, where $\{\xi_{k/2}\}_{k\ge 0}$ are independent random variables satisfying $\EE[\xi_{k/2}]=0$ and $\EE[\|\xi_{k/2}\|^2]= \sigma_{k/2}^2$ for all $k\ge 0$. Then, for the $L$-Lipschitz continuous and monotone operator $F$ and for any $\z_*\in\Z_*(\F)$, the sequence $\{\z_k\}_{k\ge 0}$ generated by S-FEG satisfies
	\begin{align}\label{eq:stoc_ub}
	\EE[\|\F\z_k\|^2]\le \frac{4L^2\|\z_0-\z_*\|^2}{k^2} + \frac{6}{k^2}\left[\sigma_0^2 + \sum_{l=1}^{k-1}(l^2\sigma_l^2 + (l+1)^2\sigma_{l+1/2}^2)\right]
	\end{align}
	for all $k\ge 1$. Furthermore, if 
	$\sigma_0^2 \le \frac{\epsilon}{6}$,
	$\sigma_k^2\le \frac{\epsilon}{6k}$ and $\sigma_{k+1/2}^2\le \frac{\epsilon}{6(k+1)}$ for all $k\ge 1$, then the bound \eqref{eq:stoc_ub} reduces to
	\begin{align*}
	\EE[\|\F\z_k\|^2]\le \frac{4L^2\|\z_0-\z_*\|^2}{k^2} + \epsilon
	\end{align*}
	for all $k\ge 1$.
\end{theorem}
Here,
we needed the noise variance $\sigma_{k/2}^2$ to decrease in the order of $\mathcal{O}(1/k)$ so that 
the stochastic error of the S-FEG 
does not accumulate.
Otherwise, 
if $\sigma_{k/2}^2$ is a constant for all $k$, 
the error accumulates with rate $\mathcal{O}(k)$.
In short, the S-FEG will suffer from error accumulation, 
unless the stochastic error decreases with rate $\mathcal{O}(1/k)$.
Such error accumulation behavior 
also appears
in a stochastic version of Nesterov's fast gradient method
\cite{nesterov:83:amf,nesterov:05:smo} for smooth convex minimization
\cite{devolder:11:sfo,ghadimi:16:agm}.
Similar to~\cite{devolder:11:sfo},
we believe that 
adjusting the step coefficients of the S-FEG
can make the S-FEG become relatively stable
even with a constant noise,
which we leave as future work.
%


\comment{
\begin{theorem}\label{thm:convergence_stoc}
	For the $L$-Lipschitz continuous and monotone operator $F$, let $\sigma_k^2\le \frac{\epsilon}{6k}$ and $\sigma_{k+1/2}^2\le \frac{\epsilon}{6(k+1)}$ for all $k=0,1,\ldots$. Then the sequence $\{\z_k\}_{k\ge 0}$ obtained by FEG satisfies
	\begin{align*}
	\EE[\|\F\z_k\|^2]\le \frac{4L^2\|\z_0-\z_*\|^2}{k^2} + \epsilon
	\end{align*}
	for all $k\ge 1$.
\end{theorem}
}

\section{Convergence analysis with nonincreasing potential lemma}\label{sec:conv_analysis}
We analyze FEG and FEG-A by
finding a nonincreasing potential function in a form $V_k=a_k\|\F\z_k\|^2-b_k\inprod{\F\z_k}{\z_0-\z_k}$
in the lemma below.
We provide a similar potential lemma for S-FEG
in Appendix~\ref{appx:s_feg_pot}.
The convergence analyses of 
EAG and Halpern iteration
are also based on such potential function
\cite{diakonikolas:20:hif,yoon:21:aaf}. 

\begin{lemma}\label{lem:potential}
	Let $\{\z_k\}_{k\ge 0}$ be the sequence generated by \eqref{alg:special}
	with $\{\alpha_k\}_{k\ge 0}$, $\{\beta_k\}_{k\ge 0}$,
	    $\{L_k\}_{k\ge 0}\subset (0,\infty)$ and $\{\rho_k\}_{k\ge 0}\subset\reals$,
	   satisfying $\alpha_0\in(0,\infty)$, $\alpha_k\in\big(0,\frac{1}{L_k}\big]$, $\beta_0 = 1$, $\{\beta_k\}_{k\ge 1}\subseteq(0,1)$ for all $k\ge 1$, and
	\begin{align*}
	\frac{(1-\beta_{k+1})}{2\beta_{k+1}}(\alpha_{k+1}+2\rho_{k+1})-\rho_{k+1}
	\le \frac{1}{2\beta_k}(\alpha_k+2\rho_k)-\rho_k
	\end{align*}
	for all $k\ge 0$. 
	Assume that 
	the following conditions are satisfied.
	\begin{align*}
	\|\F\z_1-\F\z_0\| &\le L_0\|\z_1-\z_0\| \\
	\|\F\z_{k+1}-\F\z_{k+1/2}\| &\le L_k \|\z_{k+1}-\z_{k+1/2}\| \quad\text{for all $k\ge 1$}, 
	\\
	\inprod{\F\z_{k+1}-\F\z_k}{\z_{k+1}-\z_k} &\ge \rho_k\|\F\z_{k+1}-\F\z_k\|^2 \quad\text{for all $k\ge 1$.}
	\end{align*}
Then the potential function
	\begin{align*}
	V_k = a_k\|\F\z_k\|^2-b_k\inprod{\F\z_k}{\z_0-\z_k}
	\end{align*}
	with $a_0 = \frac{\alpha_0(L_0^2\alpha_0^2-1)}{2}$, $b_0 = 0$, $b_1 = 1$,
	\begin{align*}
	a_k = \frac{b_k(1-\beta_k)}{2\beta_k}(\alpha_k+2\rho_k) - b_k\rho_k \quad\text{and}\quad
	b_{k+1} = \frac{b_k}{1-\beta_k}
	\end{align*}
	for all $k\ge 1$ satisfies $V_k\le V_{k-1}$ for all $k\ge 1$.
\end{lemma}


Based on the above potential lemma,
we next provide a convergence analysis of FEG.
The analyses for the convergence rate of FEG-A and S-FEG, \ie, the proofs of Theorem~\ref{thm:convergence_adaptive} and Theorem~\ref{thm:convergence_stoc}, are similar to that of FEG and are provided in Appendix~\ref{appx:feg_a_thm} and Appendix~\ref{appx:s_feg_thm}.
\comment{
Since we use $\beta_k = \frac{1}{k+1}$ for FEG and FEG-A, we get
\begin{gather*}
b_k = \frac{1}{1-\beta_{k-1}}b_{k-1} = \Big(\prod_{i=1}^{k-1}\frac{1}{1-\beta_i}\Big)b_1 = k
\end{gather*}
for all $k\ge 0$ in Lemma~\ref{lem:potential} for the both cases.
}

\subsection{Convergence analysis for FEG}
\textit{Proof of Theorem~\ref{thm:convergence_deterministic}.}
	Recall that 
	FEG is equivalent to \eqref{alg:special} with $\alpha_k = \frac{1}{L}$, $\beta_k = \frac{1}{k+1}$, and $\rho_k = \rho$.
	It is straightforward to verify that the given $\{\alpha_k\}_{k\ge 0}$ and $\{\beta_k\}_{k\ge 0}$ satisfy the conditions in Lemma~\ref{lem:potential} with $L_k = L$ for all $k\ge 0$. Since
	\begin{gather*}
	a_k = \frac{b_k(1-\beta_k)}{2\beta_k}(\alpha_k+2\rho_k) - b_k\rho_k = \frac{k^2}{2}\Big(\frac{1}{L}+2\rho\Big)-k\rho \qquad\text{and}\\
	b_k = \frac{1}{1-\beta_{k-1}}b_{k-1} = \Big(\prod_{i=1}^{k-1}\frac{1}{1-\beta_i}\Big)b_1 = k,
	\end{gather*}
	Lemma~\ref{lem:potential} implies that
	\begin{align*}
	0=V_0\ge V_k = \left(\frac{k^2}{2}\Big(\frac{1}{L}+2\rho\Big)-k\rho\right)\|\F\z_k\|^2 - k\inprod{\F\z_k}{\z_0-\z_k}.
	\end{align*}
	Therefore,
	\begin{align*}
	\frac{k^2}{2}\Big(\frac{1}{L}+2\rho\Big)\|\F\z_k\|^2&\le k\inprod{\F\z_k}{\z_0-\z_k} + k\rho\|\F\z_k\|^2\\
	&= k\inprod{\F\z_k}{\z_0-\z_*} +k\inprod{\F\z_k}{\z_*-\z_k} + k\rho\|\F\z_k\|^2\\
	&\le k\inprod{\F\z_k}{\z_0-\z_*} \qquad(\because \text{$\rho$-comonotonicity of $\F$})\\
	&\le k\|\F\z_k\| \|\z_0-\z_*\|.
	\end{align*}
	The desired result follows directly by dividing both sides by $\frac{k^2}{2}\big(\frac{1}{L}+2\rho\big)\|\F\z_k\|$.
\qed

\comment{
\subsection{Convergence analysis for FEG-A}
\textit{Proof of Theorem~\ref{thm:convergence_adaptive}.}
	Note that 
	FEG-A is equivalent to \eqref{alg:special} with $\alpha_k = \tau_k$ and $\beta_k = \frac{1}{k+1}$.
	The given sequence in FEG-A satisfies the conditions in Lemma~\ref{lem:potential} with $L_k = \frac{1}{\tau_k}$ and $\rho_k = \frac{\eta_k-\tau_k}{2}$:
	\begin{align*}
	\frac{(1-\beta_{k+1})}{2\beta_{k+1}}(\alpha_{k+1}+2\rho_{k+1})-\rho_{k+1}= \frac{k}{2}\eta_{k+1} + \frac{1}{2}\tau_{k+1}\le \frac{k}{2}\eta_{k} + \frac{1}{2}\tau_{k}= \frac{1}{2\beta_k}(\alpha_k + 2\rho_k) - \rho_k
	\end{align*}
	where the inequality follows from the fact that $\{\tau_k\}_{k\ge 0}$ and $\{\eta_k\}_{k\ge 0}$ are nonincreasing sequences. Since
	\begin{gather*}
	a_k = \frac{b_k(1-\beta_k)}{2\beta_k}(\alpha_k+2\rho_k) - b_k\rho_k = \frac{k}{2}((k-1)\eta_k+\tau_k) \qquad\text{and}\\
	b_k = \frac{1}{1-\beta_{k-1}}b_{k-1} = \Big(\prod_{i=1}^{k-1}\frac{1}{1-\beta_i}\Big)b_1 = k,
	\end{gather*}
	Lemma~\ref{lem:potential} implies that
	\begin{align*}
	0 = V_1 \ge V_k = \frac{k}{2}((k-1)\eta_k+\tau_k)\|\F\z_k\|^2 - k\inprod{\F\z_k}{\z_0-\z_k}.
	\end{align*}
	Therefore, 
	\begin{align*}
	\frac{k}{2}((k-1)\eta_k+\tau_k+2\rho)\|\F\z_k\|^2 &\le k\inprod{\F\z_k}{\z_0-\z_k}+k\rho\|\F\z_k\|^2\\
	&= k\inprod{\F\z_k}{\z_0-\z_*} +k\inprod{\F\z_k}{\z_*-\z_k} +k\rho\|\F\z_k\|^2\\
	&\le k\inprod{\F\z_k}{\z_0-\z_*} \qquad(\because \text{$\rho$-comonotonicity of $\F$})\\
	&\le k\|\F\z_k\| \|\z_0-\z_*\|.
	\end{align*}
	Then by dividing both sides by $\frac{k}{2}((k-1)\eta_k+\tau_k+2\rho)\|\F\z_k\|$ and using Lemma~\ref{lem:lower_bounds_of_stepsizes}, we get
	\begin{align*}
	\|\F\z_k\|\le \frac{2\|\z_0-\z_*\|}{(k-1)\eta_k+\tau_k+2\rho}\le \frac{2\|\z_0-\z_*\|}{((k-1)(1-\delta)+1)\Big(\frac{1-\delta}{L}+2\rho\Big)}.
	\end{align*}
\qed

\subsection{Convergence analysis for S-FEG}
S-FEG also has the same form of potential function, $V_k=a_k\|\F\z_k\|^2-b_k\inprod{\F\z_k}{\z_0-\z_k}$, as that of FEG and FEG-A but with taking expectation on it. The proof of Theorem~\ref{thm:convergence_stoc} is built upon a potential lemma similar to that of Lemma~\ref{lem:potential}. We postpone the detailed proof of Theorem~\ref{thm:convergence_stoc} to Appendix~\ref{appx:s_feg_thm}.
}

\section{Discussion: first-order methods for Lipschitz continuous operators}
\label{sec:disc}

Throughout this paper, we studied and constructed efficient methods
in a class of first-order methods:
\begin{align*}
\z_k \in \z_0 + \text{span}\{\F\z_0,\cdots,\F\z_k\}
\end{align*}
denoted by $\mathcal{A}$,
for smooth structured nonconvex-nonconcave problems.
We observed that all existing first-order methods, including the FEG,
required an additional condition,
such as the negative comonoticity,
on a Lipschitz continuous $\F$
to guarantee convergence.
One would then be curious whether or not there exists
an (efficient) method in class $\mathcal{A}$
that guarantees convergence without any additional condition
on a Lipschitz continuous $\F$.
Unfortunately, the following lemma
states that there exists a \emph{worst-case}\footnote{
\cite{hsieh:21:tlo,letcher:21:oti}
also introduce worst-case minimax examples
that existing methods cannot find a stationary point.
A key difference from our example
is that their saddle-gradient operators
are not Lipschitz continuous.
In addition, the considered classes of methods
in \cite{hsieh:21:tlo,letcher:21:oti}
exclude EG+ and FEG, unlike the class $\mathcal{A}$.
} 
smooth example
that none of the methods in $\mathcal{A}$
can find its stationary point.
The corresponding smooth function is illustrated 
in Figure~\ref{fig:smooth_worst}.

\begin{lemma}\label{lem:smooth_worst}
Let us consider the following 
function
$f:\reals^2\rightarrow \reals$ for some $L,R>0$:
\begin{align}\label{eq:smooth_worst}
f(x,y)=\begin{cases}
	\frac{R}{2} &\text{for }x<y-\sqrt{\frac{R}{L}}\\
	-\frac{L}{2}(x-y)^2 - \sqrt{LR}(x-y) &\text{for }y-\sqrt{\frac{R}{L}}\le x<y\\
	\frac{L}{2}(x-y)^2-\sqrt{LR}(x-y) &\text{for }y
	\le x<y+\sqrt{\frac{R}{L}}\\
	-\frac{R}{2} &\text{for }y+\sqrt{\frac{R}{L}}<x.
	\end{cases}
\end{align}
Its saddle-gradient operator $\F$
is $L$-Lipschitz continuous
but not comonotone.\footnote{
Let $\z = \big(\x,\x+\sqrt{\frac{R}{L}}\big)$
and $\w = (0,0)$.
Since $\F\z = (0,0)$ and $\F\w = (-\sqrt{LR},-\sqrt{LR})$,
we get $\inprod{\F\z - \F\w}{\z - \w} = 2\sqrt{LR}\x + R$
and $\|\F\z - \F\w\|^2 = 2LR$,
which implies that $\rho = -\infty$ 
in the comonotonicity condition
as $\x\to-\infty$.} 
Then,
the sequence $\{\z_k\}_{k\ge 0}$ generated by any first-order method 
in class $\mathcal{A}$
with $\z_0=(0,0)$ satisfies $\|\F\z_k\|^2 = 2LR$ for all $k\ge 0$.
\end{lemma}
\begin{proof}
$\F$ satisfies $\F(x,y) = (-\sqrt{LR},-\sqrt{LR})$ whenever $x=y$.
Hence, for all sequences $\{\z_k\}_{k\ge 0}$ satisfying $\z_0 = (0,0)$ and $\z_k \in \z_0 + \text{span}\{\F\z_0,\cdots,\F\z_k\}$ for all $k\ge 0$, we have that $\{\z_k\}_{k\ge 0}\subseteq \{\z=(x,y)\in\reals^2|x=y\}$; thus, $\|\F\z_k\|^2 = 2LR$ for all $k\ge 0$. 
\end{proof}

The lemma implies that
one should consider a class of methods, other than the class $\mathcal{A}$,
to guarantee finding a stationary point of any smooth problem,
which we leave as future work.
We also leave finding additional conditions
for a Lipschitz continuous $\F$,
weaker than the weak MVI condition and the negative comonotonicity
(with $\rho>-\frac{1}{2L}$),
which guarantee convergence or its accelerated rate, respectively, as future work.

\comment{
Therefore, finding the weaker additional condition than the weak MVI condition, which makes a first-order method in $\mathcal{A}$ converge for the smooth minimax problem, would be an interesting problem.
Also, to develop a method which converges under only the Lipschitz continuity assumption on $\F$, we need to find another class of 
methods, which is not $\mathcal{A}$.
We leave solving those problems as future work.}

\begin{figure}
    \centering
    \includegraphics[width=.5\textwidth]{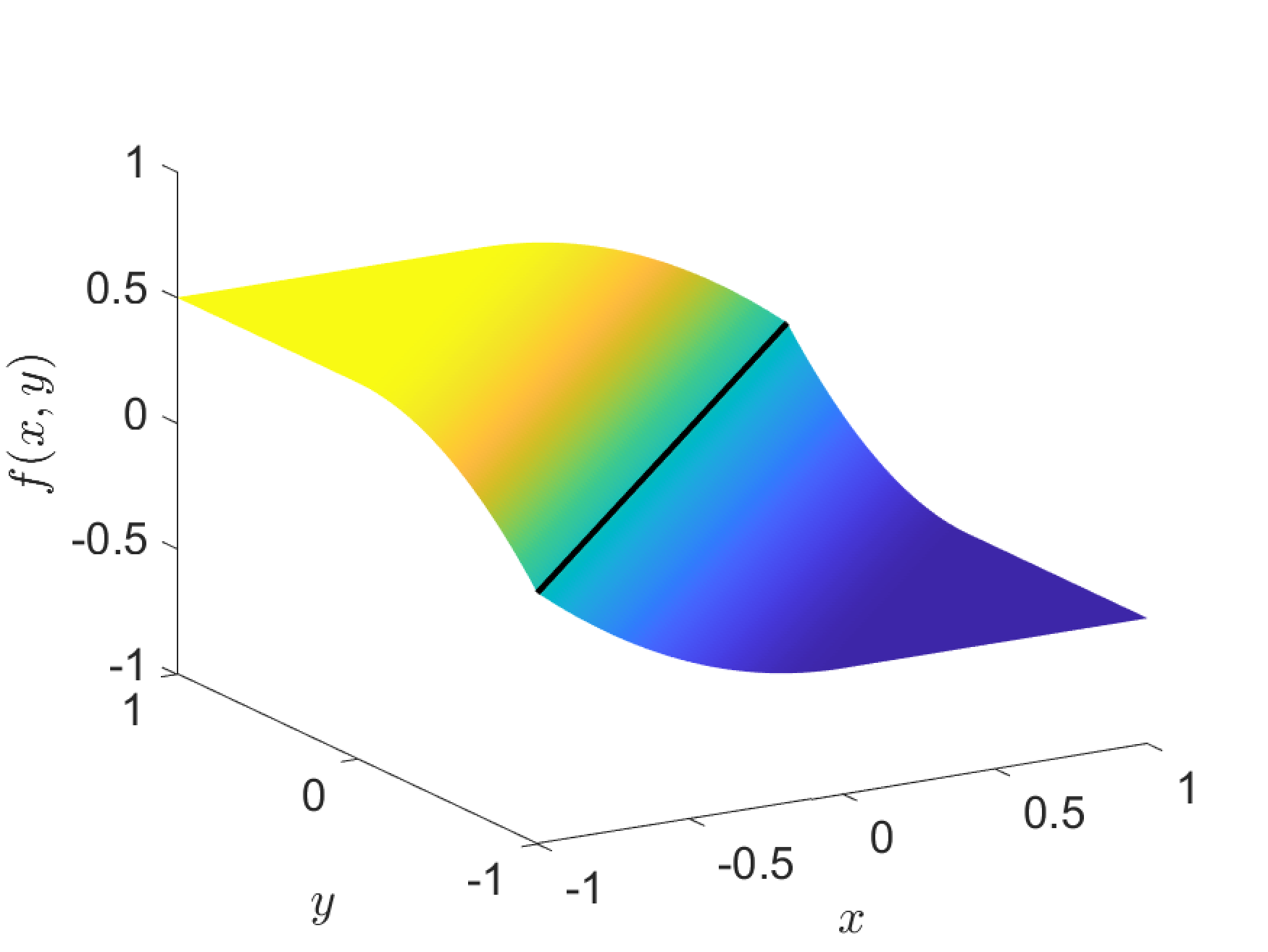}
    \caption{A smooth worst-case example $f(x,y)$~\eqref{eq:smooth_worst} with $L=R=1$ for first-order methods. any sequence $\{\z_k\}_{k\ge 0}$ generated by a first-order method in class $\mathcal{A}$ starting from $(0,0)$ is contained in the line $x=y$.}
    \label{fig:smooth_worst}
\end{figure}
\section{Conclusion}
This paper proposed a two-time-scale and anchored extragradient method, named FEG, for smooth structured nonconvex-nonconcave problems. The proposed FEG has an accelerated $\mathcal{O}(1/k^2)$ rate, 
with respect to the squared gradient norm,
for the Lipschitz continuous and negative comonotone operators for the first time. 
The FEG also 
has value for smooth convex-concave problems, 
compared to existing works.
We further studied its backtracking line-search version, named FEG-A, for the smooth structured nonconvex-nonconcave problems and studied its stochastic version, named S-FEG, for smooth convex-concave problems. We leave 
extending this work to 
stochastic, composite, or more general nonconvex-nonconcave setting
and applying to more realistic problems
as future work.
\comment{
\begin{itemize}
    \item stochastic FEG under negative comonotone condition is unclear.
    \item composite analysis
    \item relation with ODE?
    \item strong monotone? adaptive restart?
\end{itemize}
}

\begin{ack}

This work was supported in part by the National Research Foundation of Korea (NRF) grant funded by the Korea government (MSIT) (No. 2019R1A5A1028324), the POSCO Science Fellowship of POSCO TJ Park Foundation, and the Samsung Science and Technology Foundation
(No. SSTF-BA2101-02).
\end{ack}

\bibliographystyle{plain}
\bibliography{master,mastersub}

\comment{
\section*{Checklist}

\comment{
The checklist follows the references.  Please
read the checklist guidelines carefully for information on how to answer these
questions.  For each question, change the default \answerTODO{} to \answerYes{},
\answerNo{}, or \answerNA{}.  You are strongly encouraged to include a {\bf
justification to your answer}, either by referencing the appropriate section of
your paper or providing a brief inline description.  For example:
\begin{itemize}
  \item Did you include the license to the code and datasets? \answerYes{See Section~\ref{gen_inst}.}
  \item Did you include the license to the code and datasets? \answerNo{The code and the data are proprietary.}
  \item Did you include the license to the code and datasets? \answerNA{}
\end{itemize}
Please do not modify the questions and only use the provided macros for your
answers.  Note that the Checklist section does not count towards the page
limit.  In your paper, please delete this instructions block and only keep the
Checklist section heading above along with the questions/answers below.
}

\begin{enumerate}

\item For all authors...
\begin{enumerate}
  \item Do the main claims made in the abstract and introduction accurately reflect the paper's contributions and scope?
    \answerYes{}
  \item Did you describe the limitations of your work?
    \answerYes{
    See Section~\ref{sec:stochastic} and~\ref{sec:disc}.}
  \item Did you discuss any potential negative societal impacts of your work?
    \answerNA{}
  \item Have you read the ethics review guidelines and ensured that your paper conforms to them?
    \answerYes{}
\end{enumerate}

\item If you are including theoretical results...
\begin{enumerate}
  \item Did you state the full set of assumptions of all theoretical results?
    \answerYes{See Section~\ref{sec:preliminaries}}
	\item Did you include complete proofs of all theoretical results?
    \answerYes{See Section~\ref{sec:conv_analysis} and Appendix}
\end{enumerate}

\item If you ran experiments...
\begin{enumerate}
  \item Did you include the code, data, and instructions needed to reproduce the main experimental results (either in the supplemental material or as a URL)?
    \answerNA{We did not run experiments}
  \item Did you specify all the training details (e.g., data splits, hyperparameters, how they were chosen)?
    \answerNA{}
	\item Did you report error bars (e.g., with respect to the random seed after running experiments multiple times)?
    \answerNA{}
	\item Did you include the total amount of compute and the type of resources used (e.g., type of GPUs, internal cluster, or cloud provider)?
    \answerNA{}
\end{enumerate}

\item If you are using existing assets (e.g., code, data, models) or curating/releasing new assets...
\begin{enumerate}
  \item If your work uses existing assets, did you cite the creators?
    \answerNA{}
  \item Did you mention the license of the assets?
    \answerNA{}
  \item Did you include any new assets either in the supplemental material or as a URL?
    \answerNA{}
  \item Did you discuss whether and how consent was obtained from people whose data you're using/curating?
    \answerNA{}
  \item Did you discuss whether the data you are using/curating contains personally identifiable information or offensive content?
    \answerNA{}
\end{enumerate}

\item If you used crowdsourcing or conducted research with human subjects...
\begin{enumerate}
  \item Did you include the full text of instructions given to participants and screenshots, if applicable?
    \answerNA{}
  \item Did you describe any potential participant risks, with links to Institutional Review Board (IRB) approvals, if applicable?
    \answerNA{}
  \item Did you include the estimated hourly wage paid to participants and the total amount spent on participant compensation?
    \answerNA{}
\end{enumerate}

\end{enumerate}
}
\newpage

\appendix
\renewcommand\thesubsection{\Alph{subsection}}
\numberwithin{lemma}{subsection}

\section*{Appendix}

\subsection{Proof for Section~\ref{sec:preliminaries}}
\subsubsection{Proof of Example~\ref{ex:comonotone,relation}}
\label{appx:comonotone,relation}
Let $f_\eta$ be the saddle envelope of $f$~\cite{grimmer:20:tlo}:
\begin{align*}
f_\eta(\bar{\x},\bar{\y}):= \min_{\x\in\Xs}\max_{\y\in\Ys} f(\x,\y)+\frac{\eta}{2}\|\x-\bar{\x}\|^2 - \frac{\eta}{2}\|\y-\bar{\y}\|^2
,\end{align*}
and 
$\F_\eta$ be its saddle gradient operator.
Proposition 2.10 in \cite{grimmer:20:tlo} shows that $f_\eta$ satisfies
\begin{gather*}
\frac{\eta\alpha}{\eta+\alpha}\I \preceq \nabla_{\x\x}^2 f_\eta \preceq \eta\I
\quad\text{and}\quad
\frac{\eta\alpha}{\eta+\alpha}\I \preceq -\nabla_{\y\y}^2 f_\eta \preceq \eta\I.
\end{gather*}
This implies that
$\F_\eta$ is $\frac{\eta\alpha}{\eta+\alpha}$-strongly monotone
(and thus monotone).

It is enough to show that $\F_\eta$ is monotone if and only if $\F$ is $-\frac{1}{\eta}$-comonotone. 
By Lemma~2.5 in \cite{grimmer:20:tlo}, we have the relationship
$\F_\eta\z= \F\R\z$,
where 
$\R:=\big(\I + \frac{1}{\eta}\F\big)^{-1}$
denotes the standard resolvent of $\frac{1}{\eta}\F$.
The resolvent $\R$ is injective for $\eta>\gamma$.
Let $Z:= \Xs\times\Ys$.
Then, $\F_\eta$ is monotone if and only if
\begin{align*}
& \Inprod{\F\R\z
	-\F\R\w
}{\z-\w}\ge 0, \qquad\forall \z,\w\in Z,\\
\Leftrightarrow \quad& \Inprod{\F\z'-\F\w'}{\Big(\I+\frac{1}{\eta}\F\Big)\z'-\Big(\I+\frac{1}{\eta}\F\Big)\w'}\ge 0, \qquad
\forall \z,\w\in Z,
\z'=\R\z,
\w'= \R\w, \\
\Leftrightarrow \quad& \inprod{\F\z'-\F\w'}{\z'-\w'}\ge -\frac{1}{\eta}\|\F\z'-\F\w'\|^2, \qquad
\forall \z',\w'\in Z,
\end{align*}
which corresponds to the $-\frac{1}{\eta}$-comonotonicity of $\F$.
\qed

\subsection{Proof for Section~\ref{sec:FEG}}

\comment{
\subsubsection{Proof of Theorem~\ref{thm:convergence_deterministic}}
It is straightforward to verify that the given $\{\alpha_k\}_{k\ge 0}$ and $\{\beta_k\}_{k\ge 0}$ satisfy the conditions in Lemma~\ref{lem:potential}. By noting that
\begin{gather*}
b_k = \frac{1}{1-\beta_{k-1}}b_{k-1} = \Big(\prod_{i=1}^{k-1}\frac{1}{1-\beta_i}\Big)b_1 = k \qquad\text{and}\\
a_k = \frac{b_k(1-\beta_k)}{2\beta_k}(\alpha_k+2\rho) - b_k\rho = \frac{k^2}{2}\Big(\frac{1}{L}+2\rho\Big)-k\rho,
\end{gather*}
Lemma~\ref{lem:potential} implies that
\begin{align*}
0=V_1\ge V_k = \left(\frac{k^2}{2}\Big(\frac{1}{L}+2\rho\Big)-k\rho\right)\|\F\z_k\|^2 - k\inprod{\F\z_k}{\z_0-\z_k}.
\end{align*}
Therefore,
\begin{align*}
\frac{k^2}{2}\Big(\frac{1}{L}+2\rho\Big)\|\F\z_k\|^2&\le k\inprod{\F\z_k}{\z_0-\z_k} + k\rho\|\F\z_k\|^2\\
&= k\inprod{\F\z_k}{\z_0-\z_*} +k\inprod{\F\z_k}{\z_*-\z_k} + k\rho\|\F\z_k\|^2\\
&\le k\inprod{\F\z_k}{\z_0-\z_*}\\
&\le k\|\F\z_k\| \|\z_0-\z_*\|.
\end{align*}
The desired result follows directly by dividing both sides by $\frac{k^2}{2}\Big(\frac{1}{L}+2\rho\Big)\|\F\z_k\|$.
\qed
}
\subsubsection{Proof of Example~\ref{lem:worst_case}}
\label{appx:worst_case}
Starting from $\z_0 = (1,0)$,
it is easy to verify that 
$\z_{1/2}=(1,0)$, $\z_1=(1,1)$, 
$\z_{1+1/2} = \big(\frac{1}{2},1\big)$,
and $\z_2 = (0,1)$. 
We next use the induction to show that 
$\z_k = \big(0,\frac{2}{k}\big)$
for $k=4l+2$ and for all $l=0,1,2,\ldots$.
Assume that $\z_k = \big(0,\frac{2}{k}\big)$ for some $k=4l+2$.
Then, the next eight consecutive iterates are as follows:
\begingroup
\allowdisplaybreaks
\begin{align*}
	\z_{k+1/2} &= \z_k + \frac{1}{k+1}(\z_0-\z_k) - \Big(1-\frac{1}{k+1}\Big)\frac{1}{L}\F\z_k\\
	&= \Big(\frac{1}{k+1},\frac{2}{k+1}\Big) - \frac{k}{k+1}\Big(\frac{2}{k},0\Big) = \Big(-\frac{1}{k+1},\frac{2}{k+1}\Big), \\
	\z_{k+1} &= \z_k + \frac{1}{k+1}(\z_0-\z_k) - \frac{1}{L}\F\z_{k+1/2}\\
	&= \Big(\frac{1}{k+1},\frac{2}{k+1}\Big) - \Big(\frac{2}{k+1},\frac{1}{k+1}\Big) = \Big(-\frac{1}{k+1},\frac{1}{k+1}\Big), \\
	\z_{k+1+1/2} &= \z_{k+1} + \frac{1}{k+2}(\z_0-\z_{k+1}) - \Big(1-\frac{1}{k+2}\Big)\frac{1}{L}\F\z_{k+1}\\
	&= \Big(0,\frac{1}{k+2}\Big) -\frac{k+1}{k+2}\Big(\frac{1}{k+1},\frac{1}{k+1}\Big) = \Big(-\frac{1}{k+2},0\Big), \\
	\z_{k+2} &= \z_{k+1} + \frac{1}{k+2}(\z_0-\z_{k+1}) - \frac{1}{L}\F\z_{k+1+1/2}\\
	&=\Big(0,\frac{1}{k+2}\Big) -\Big(0,\frac{1}{k+2}\Big)= (0,0)\\
	\z_{k+2+1/2} &= \z_{k+2}+\frac{1}{k+3}(\z_0-\z_{k+2}) - \Big(1-\frac{1}{k+3}\Big)\frac{1}{L}\F\z_{k+2}\\
	&= \Big(\frac{1}{k+3},0\Big), \\
	\z_{k+3} &= \z_{k+2}+\frac{1}{k+3}(\z_0-\z_{k+2}) - \frac{1}{L}\F\z_{k+2+1/2}\\
	&= \Big(\frac{1}{k+3},0\Big) - \Big(0,-\frac{1}{k+3}\Big) = \Big(\frac{1}{k+3},\frac{1}{k+3}\Big), \\
	\z_{k+3+1/2} &= \z_{k+3} + \frac{1}{k+4}(\z_0-\z_{k+3}) - \Big(1-\frac{1}{k+4}\Big)\frac{1}{L}\F\z_{k+3}\\
	&= \Big(\frac{2}{k+4},\frac{1}{k+4}\Big) -\frac{k+3}{k+4}\Big(\frac{1}{k+3},-\frac{1}{k+3}\Big) = \Big(\frac{1}{k+4},\frac{2}{k+4}\Big), \\
	\z_{k+4} &=	\z_{k+3} + \frac{1}{k+4}(\z_0-\z_{k+3}) - \frac{1}{L}\F\z_{k+3+1/2}\\
	&=\Big(\frac{2}{k+4},\frac{1}{k+4}\Big) - \Big(\frac{2}{k+4},-\frac{1}{k+4}\Big)= \Big(0,\frac{2}{k+4}\Big),
\end{align*}
\endgroup
so $\z_{4l+6} = \big(0,\frac{2}{4l+6}\big)$.
Therefore, 
we get
$\z_{4l+2} 
=\big(0,\frac{1}{2l+1}\big)$ for all $l\ge 0$.
\qed

\subsection{Algorithm and proofs for Section~\ref{sec:FEG-A}}

\subsubsection{A detailed description of FEG-A}\label{appx:feg_a_detailed}

A detailed description of 
the FEG-A, in Algorithm~\ref{alg:feg_a},
is provided in Algorithm~\ref{alg:feg_a_detailed}.

\begin{algorithm}[b!]
	\caption{Fast extragradient method with adaptive step size (FEG-A)}
	\label{alg:feg_a_detailed}
	\begin{algorithmic}
		\State {\bf Input:} $\z_0\in\reals^d$, 
		$\tau_{-1}\in(\max\{0,-2\rho\},\infty)$, $\eta_0\in(0,\infty)$,
		$\delta\in(0,1)$
		\State Find the smallest nonnegative integer $i_0$ such that $\hat{\z} = \z_0 - 
		\tau_{-1}(1-\delta)^{i_0}\F\z_0$ satisfies
		$
		\|\F\hat{\z}-\F\z_0\| \le \frac{1}{\tau_{-1}(1-\delta)^{i_0}}\|\hat{\z}-\z_0\|$. 
		\State $\tau_0 = \tau_{-1}(1-\delta)^{i_0}$, $\z_1 = \z_0 - \tau_0\F\z_0$.
		\For{$k=1,2,\ldots$}
		\State $i_k=j_k=0$, $\text{searching}=\text{True}$
		\While{$\text{searching}=\text{True}$}
		\State $\text{searching}=\text{False}$, $\tau_k = \tau_{k-1}(1-\delta)^{i_k}$, $\eta_k = \eta_{k-1}(1-\delta)^{j_k}$
		\begin{align*}
		\z_{k+1/2} &= \z_k + \frac{1}{k+1}(\z_0-\z_k) - \Big(1-\frac{1}{k+1}\Big)\eta_k \F\z_k \quad\text{and}\\
		\z_{k+1} &= \z_k + \frac{1}{k+1}(\z_0-\z_k) - \tau_k\F\z_{k+1/2} - \Big(1-\frac{1}{k+1}\Big)(\eta_k-\tau_k)\F\z_k
		\end{align*}
		\If{$\|\F\z_{k+1}-\F\z_{k+1/2}\| > \frac{1}{\tau_k}\|\z_{k+1}-\z_{k+1/2}\|$}
		\State $i_k \leftarrow i_k+1$
		\State $\text{searching}=\text{True}$
		\EndIf
		\If{$\inprod{\F\z_{k+1}-\F\z_k}{\z_{k+1}-\z_k} < \frac{\eta_k-\tau_k}{2}\|\F\z_{k+1}-\F\z_k\|^2$}
		\State $j_k \leftarrow j_k +1$
		\State $\text{searching}=\text{True}$
		\EndIf
		\EndWhile
		\EndFor
	\end{algorithmic}
\end{algorithm}
\subsubsection{Proof of Lemma~\ref{lem:lower_bounds_of_stepsizes}}
We show that $\tau_k \ge \underline{\tau} := \min\big\{\tau_{-1},\frac{1-\delta}{L}\big\}$ 
for all $k\ge 0$,
and $\eta_k \ge \min\{\eta_0,(1-\delta)(\tau_k+2\rho)\}$
for all $k\ge1$
by contradiction.
Note that since $\tau_{-1} > \max\{0,-2\rho\}$ and $\rho > -\frac{1-\delta}{2L}$,
both $\underline{\tau}$ and $\underline{\eta}$ are positive.

First, suppose that $\tau_k < \underline{\tau}$ 
for some $k\ge 0$. 
(1) For the case $\tau_{-1}\le \frac{1-\delta}{L}$, 
we get $\tau_k = \tau_{-1}$ for all $k\ge0$ 
by the definition of $\tau_k$, which contradicts to the assumption
$\tau_k < 
\tau_{-1}$.
(2) Consider 
the case $\tau_{-1} > \frac{1-\delta}{L}$,
where the assumption 
reduces
to $\tau_k < \frac{1-\delta}{L}$.
For $k=0$,
by the definition of $\tau_0$, we get $\|\F\hat{\z}_1-\F\z_0\|>\frac{1-\delta}{\tau_0}\|\hat{\z}_1-\z_0\|$ where $\hat{\z}_1 = \z_0-\frac{\tau_0}{1-\delta}\F\z_0$, which contradicts to the $L$-Lipschitz continuity of $\F$ as $\frac{1-\delta}{\tau_0}
\;>\; L$.
For $k\ge1$,
by the definition of $\tau_k$, there exists $i\le k$ such that 
the two corresponding iterates
\begin{align*}
\hat{\z}_{i+1/2} &= \z_i + \frac{1}{i+1}(\z_0-\z_i) - \Big(1-\frac{1}{i+1}\Big)\hat{\eta}_i\F\z_i\quad\text{and}\\
\hat{\z}_{i+1} &= \z_i + \frac{1}{i+1}(\z_0-\z_i) - \frac{\tau_k}{1-\delta}\F\hat{\z}_{i+1/2} - \Big(1-\frac{1}{i+1}\Big)\Big(\hat{\eta}_i-\frac{\tau_k}{1-\delta}\Big)\F\z_i
\end{align*}
satisfy
$\|\F\hat{\z}_{i+1}-\F\hat{\z}_{i+1/2}\|>\frac{1-\delta}{\tau_k}\|\hat{\z}_{i+1}-\hat{\z}_{i+1/2}\|$ for some 
$\hat{\eta}_i>0$.
However, this inequality contradicts to the $L$-Lipschitz continuity of $\F$ as $\frac{1-\delta}{\tau_k} 
\;>\; L$. 
Therefore, 
we have $\tau_k \ge \underline{\tau} > 0$
for all $k\ge 0$.

Similarly, suppose that $\eta_k < \min\{\eta_0,(1-\delta)(\tau_k+2\rho)\}$
for some $k\ge 1$. 
(1) For the case $\eta_0 \le (1-\delta)(\tau_k+2\rho)$, we get $\eta_i = \eta_0$ for all $1\le i \le k$
by the definition of $\eta_k$, which contradicts to the
assumption $\eta_k < \eta_0$.
(2) Consider 
the case $\eta_0 > (1-\delta)(\tau_k+2\rho)$,
where the assumption reduces to
$\eta_k < (1-\delta)(\tau_k + 2\rho)$.
%
Then by the definition of $\eta_k$, there exists $i\le k$ such that 
the two corresponding iterates
\begin{align*}
\hat{\z}_{i+1/2} &= \z_i + \frac{1}{i+1}(\z_0-\z_i) - \Big(1-\frac{1}{i+1}\Big)\frac{\eta_k}{1-\delta}\F\z_i\quad\text{and}\\
\hat{\z}_{i+1} &= \z_i + \frac{1}{i+1}(\z_0-\z_i) - \hat{\tau}_i\F\hat{\z}_{i+1/2} - \Big(1-\frac{1}{i+1}\Big)\Big(\frac{\eta_k}{1-\delta}-\hat{\tau}_i\Big)\F\z_i
\end{align*}
satisfy
$\inprod{\F\hat{\z}_{i+1}-\F\hat{\z}_i}{\hat{\z}_{i+1}-\hat{\z}_i}<\frac{\frac{\eta_k}{1-\delta}-\hat{\tau}_i}{2}\|\F\hat{\z}_{i+1}-\F\hat{\z}_i\|^2$ for some $\hat{\tau}_i \ge \tau_i$. 
However, this inequality contradicts to the $\rho$-comonotonicity of $\F$ as
$\frac{\frac{\eta_k}{1-\delta}-\hat{\tau}_i}{2} < \frac{ 
\tau_k
+2\rho - \hat{\tau}_i}{2}\le \rho$. 
Therefore, 
we have $\eta_k \ge \min\{\eta_0,(1-\delta)(\tau_k+2\rho)\}$
for all $k\ge0$.
\qed
\subsubsection{Proof of Theorem~\ref{thm:convergence_adaptive}}\label{appx:feg_a_thm}
Note that 
FEG-A is equivalent to \eqref{alg:special} with $\alpha_k = \tau_k$, $\beta_k = \frac{1}{k+1}$, and $\rho_k = \frac{\eta_k-\tau_k}{2}$.
The given sequence in FEG-A satisfies the conditions in Lemma~\ref{lem:potential} with $L_k = \frac{1}{\tau_k}$:
\begin{align*}
\frac{(1-\beta_{k+1})}{2\beta_{k+1}}(\alpha_{k+1}+2\rho_{k+1})-\rho_{k+1}= \frac{k}{2}\eta_{k+1} + \frac{1}{2}\tau_{k+1}\le \frac{k}{2}\eta_{k} + \frac{1}{2}\tau_{k}= \frac{1}{2\beta_k}(\alpha_k + 2\rho_k) - \rho_k
\end{align*}
where the inequality follows from the fact that $\{\tau_k\}_{k\ge 0}$ and $\{\eta_k\}_{k\ge 0}$ are nonincreasing sequences. Since
\begin{gather*}
a_k = \frac{b_k(1-\beta_k)}{2\beta_k}(\alpha_k+2\rho_k) - b_k\rho_k = \frac{k}{2}((k-1)\eta_k+\tau_k) \qquad\text{and}\\
b_k = \frac{1}{1-\beta_{k-1}}b_{k-1} = \Big(\prod_{i=1}^{k-1}\frac{1}{1-\beta_i}\Big)b_1 = k,
\end{gather*}
Lemma~\ref{lem:potential} implies that
\begin{align*}
0 = V_0 \ge V_k = \frac{k}{2}((k-1)\eta_k+\tau_k)\|\F\z_k\|^2 - k\inprod{\F\z_k}{\z_0-\z_k}.
\end{align*}
Therefore, 
\begin{align*}
\frac{k}{2}((k-1)\eta_k+\tau_k+2\rho)\|\F\z_k\|^2 &\le k\inprod{\F\z_k}{\z_0-\z_k}+k\rho\|\F\z_k\|^2\\
&= k\inprod{\F\z_k}{\z_0-\z_*} +k\inprod{\F\z_k}{\z_*-\z_k} +k\rho\|\F\z_k\|^2\\
&\le k\inprod{\F\z_k}{\z_0-\z_*} \qquad(\because \text{$\rho$-comonotonicity of $\F$})\\
&\le k\|\F\z_k\| \|\z_0-\z_*\|.
\end{align*}
Then by dividing both sides by $\frac{k}{2}((k-1)\eta_k+\tau_k+2\rho)\|\F\z_k\|$ and using Lemma~\ref{lem:lower_bounds_of_stepsizes}, we get
\begin{align*}
\|\F\z_k\| &\le \frac{2\|\z_0-\z_*\|}{(k-1)\eta_k+\tau_k+2\rho}. 
\end{align*}
\qed
\comment{
\subsubsection{Proof of Theorem~\ref{thm:convergence_adaptive}}
Note that by taking $\alpha_k = \tau_k$ and $\beta_k = \frac{1}{k+1}$ in \eqref{alg:special}, we get FEG-A. The given sequence in FEG-A satisfies the conditions in Lemma~\ref{lem:potential} with $L_k = \frac{1}{\tau_k}$ and $\rho_k = \frac{\eta_k-\tau_k}{2}$:
\begin{align*}
\frac{(1-\beta_{k+1})}{2\beta_{k+1}}(\alpha_{k+1}+2\rho_{k+1})-\rho_{k+1}&= \frac{k+1}{2}(\alpha_{k+1}+2\rho_{k+1}) -\rho_{k+1}\\
&= \frac{k}{2}(\alpha_{k+1}+2\rho_{k+1}) + \frac{1}{2}\alpha_{k+1}\\
&= \frac{k}{2}\eta_{k+1} + \frac{1}{2}\tau_{k+1}\\
&\le \frac{k}{2}\eta_{k} + \frac{1}{2}\tau_{k}\\
&= \frac{1}{2\beta_k}(\alpha_k + 2\rho_k) - \rho_k
\end{align*}
where the inequality follows from the fact that $\{\tau_k\}_{k\ge 0}$ and $\{\eta_k\}_{k\ge 0}$ are nonincreasing sequences. By noting that
\begin{gather*}
b_k = \frac{1}{1-\beta_{k-1}}b_{k-1} = \Big(\prod_{i=1}^{k-1}\frac{1}{1-\beta_i}\Big)b_1 = k \qquad\text{and}\\
a_k = \frac{b_k(1-\beta_k)}{2\beta_k}(\alpha_k+2\rho_k) - b_k\rho_k = \frac{k}{2}((k-1)\eta_k+\tau_k),
\end{gather*}
Lemma~\ref{lem:potential} implies that
\begin{align*}
0 = V_1 \ge V_k = \frac{k}{2}((k-1)\eta_k+\tau_k)\|\F\z_k\|^2 - k\inprod{\F\z_k}{\z_0-\z_k}.
\end{align*}
Therefore, 
\begin{align*}
\frac{k}{2}((k-1)\eta_k+\tau_k+2\rho)\|\F\z_k\|^2 &\le k\inprod{\F\z_k}{\z_0-\z_k}+k\rho\|\F\z_k\|^2\\
&= k\inprod{\F\z_k}{\z_0-\z_*} +k\inprod{\F\z_k}{\z_*-\z_k} +k\rho\|\F\z_k\|^2\\
&\le k\inprod{\F\z_k}{\z_0-\z_*}\\
&\le k\|\F\z_k\| \|\z_0-\z_*\|.
\end{align*}
Then by dividing both sides by $\frac{k}{2}((k-1)\eta_k+\tau_k+2\rho)\|\F\z_k\|$ and using Lemma~\ref{lem:lower_bounds_of_stepsizes}, we get
\begin{align*}
\|\F\z_k\|&\le \frac{2\|\z_0-\z_*\|}{(k-1)\eta_k+\tau_k+2\rho}\\
&\le \frac{2\|\z_0-\z_*\|}{(k-1)\Big(\frac{(1-\delta)^2}{L}+(1-\delta)2\rho\Big)+\frac{1-\delta}{L}+2\rho}\\
&= \frac{2\|\z_0-\z_*\|}{((k-1)(1-\delta)+1)\Big(\frac{1-\delta}{L}+2\rho\Big)}
\end{align*}
\qed
}
\subsection{Proofs for Section~\ref{sec:conv_analysis}}
\subsubsection{Proof of Lemma~\ref{lem:potential}}
First, for $k=0$, note that
\begin{align}
V_1 &= a_1\|\F\z_1\|^2 - b_1\inprod{\F\z_1}{\z_0-\z_1}\nonumber\\
&=a_1\|\F\z_1\|^2 - \alpha_0b_1\inprod{\F\z_1}{\F\z_0}\nonumber\\
&=\Big(\frac{b_1(1-\beta_1)}{2\beta_1}(\alpha_1+2\rho_1) -b_1\rho_1\Big)\|\F\z_1\|^2 - \alpha_0\inprod{\F\z_1}{\F\z_0}\nonumber\\
&\le \Big(\frac{b_1}{2\beta_0}(\alpha_0+2\rho_0)-b_1\rho_0\Big)\|\F\z_1\|^2 - \alpha_0\inprod{\F\z_1}{\F\z_0}\nonumber\\
&= \frac{\alpha_0}{2}\|\F\z_1\|^2 - \alpha_0\inprod{\F\z_1}{\F\z_0}.\label{eq:init_potential_gap_1}
\end{align}
By the given condition, we get
\begin{align}\label{eq:init_potential_gap_2}
0\le L_0^2\|\z_1-\z_0\|^2-\|\F\z_1-\F\z_0\|^2 = L_0^2\alpha_0^2\|\F\z_0\|^2 - \|\F\z_1-\F\z_0\|^2.
\end{align}
Hence, the sum of \eqref{eq:init_potential_gap_1} and \eqref{eq:init_potential_gap_2} with multiplying factor $\frac{\alpha_0}{2}$ yields
\begin{align*}
V_1&\le \frac{\alpha_0}{2}\|\F\z_1\|^2 - \alpha_0\inprod{\F\z_1}{\F\z_0} 
+ \frac{\alpha_0}{2}(L_0^2\alpha_0^2\|\F\z_0\|^2 - \|\F\z_1-\F\z_0\|^2)\\
&= \frac{\alpha_0(L_0^2\alpha_0^2-1)}{2}\|\F\z_0\|^2 = V_0.
\end{align*}
Next, for $k\ge 1$, here we note the following relations for later use:
\begin{align*}
\z_{k+1}-\z_k &= \frac{\beta_k}{1-\beta_k}(\z_0-\z_{k+1}) - \frac{\alpha_k}{1-\beta_k}\F\z_{k+1/2}-2\rho_k\F\z_k,\\
\z_{k+1}-\z_k &= \beta_k(\z_0-\z_k) - \alpha_k\F\z_{k+1/2}-2(1-\beta_k)\rho_k\F\z_k, \text{ and}\\
\z_{k+1}-\z_{k+1/2}&= \alpha_k((1-\beta_k)\F\z_k - \F\z_{k+1/2}).
\end{align*}
Then, by the given condition, we have
\begin{align}
V_k - V_{k+1}\ge& V_k - V_{k+1} - \frac{b_k}{\beta_k}(\inprod{\F\z_{k+1}-\F\z_k}{\z_{k+1}-\z_k}-\rho_k\|\F\z_{k+1}-\F\z_k\|^2)\nonumber\\
=&V_k - V_{k+1} - \frac{b_k}{\beta_k}\inprod{\F\z_{k+1}}{\z_{k+1}-\z_k} + \frac{b_k}{\beta_k}\inprod{\F\z_k}{\z_{k+1}-\z_k} + \frac{b_k\rho_k}{\beta_k}\|\F\z_{k+1}-\F\z_k\|^2\nonumber\\
=&(a_k\|\F\z_k\|^2-b_k\inprod{\F\z_k}{\z_0-\z_k})-(a_{k+1}\|\F\z_{k+1}\|^2-b_{k+1}\inprod{\F\z_{k+1}}{\z_0-\z_{k+1}})\nonumber\\
&-\frac{b_k}{\beta_k}\Inprod{\F\z_{k+1}}{\frac{\beta_k}{1-\beta_k}(\z_0-\z_{k+1})-\frac{\alpha_k}{1-\beta_k}\F\z_{k+1/2}-2\rho_k\F\z_k}\nonumber\\
&+\frac{b_k}{\beta_k}\inprod{\F\z_k}{\beta_k(\z_0-\z_k)-\alpha_k\F\z_{k+1/2}-2(1-\beta_k)\rho_k\F\z_k}\nonumber\\
&+\frac{b_k\rho_k}{\beta_k}\|\F\z_{k+1}-\F\z_k\|^2\nonumber\\
=&\Big(a_k - \frac{b_k(1-2\beta_k)\rho_k}{\beta_k}\Big)\|\F\z_k\|^2 - \Big(-\frac{b_k\rho_k}{\beta_k} + a_{k+1}\Big)\|\F\z_{k+1}\|^2\nonumber\\
&+\Big(b_{k+1}-\frac{b_k}{1-\beta_k}\Big)\inprod{\F\z_{k+1}}{\z_0-\z_{k+1}}+\frac{b_k\alpha_k}{\beta_k(1-\beta_k)}\inprod{\F\z_{k+1}}{\F\z_{k+1/2}}\nonumber\\
&-\frac{\alpha_k b_k}{\beta_k}\inprod{\F\z_k}{\F\z_{k+1/2}}\nonumber\\
=&\Big(a_k - \frac{b_k(1-2\beta_k)\rho_k}{\beta_k}\Big)\|\F\z_k\|^2 - \Big(-\frac{b_k\rho_k}{\beta_k} + a_{k+1}\Big)\|\F\z_{k+1}\|^2\nonumber\\
&+\frac{b_k\alpha_k}{\beta_k(1-\beta_k)}\inprod{\F\z_{k+1}}{\F\z_{k+1/2}}-\frac{\alpha_k b_k}{\beta_k}\inprod{\F\z_k}{\F\z_{k+1/2}}. \qquad(\because b_{k+1}=\frac{b_k}{1-\beta_k}.)\label{eq:potential_gap_1}
\end{align}
By the given condition, we also have
\begin{align}\label{eq:potential_gap_2}
0&\ge \|\F\z_{k+1}-\F\z_{k+1/2}\|^2 - L_k^2\|\z_{k+1}-\z_{k+1/2}\|^2 \nonumber\\
&=\|\F\z_{k+1}-\F\z_{k+1/2}\|^2 - L_k^2\alpha_k^2\|(1-\beta_k)\F\z_k-\F\z_{k+1/2}\|^2.
\end{align}
Hence, the sum of \eqref{eq:potential_gap_1} and \eqref{eq:potential_gap_2} with multiplying factor $\frac{b_k}{2L_k^2\alpha_k\beta_k(1-\beta_k)}$ yields
\begin{align*}
V_k-V_{k+1}\ge& \Big(a_k - \frac{b_k(1-2\beta_k)\rho_k}{\beta_k}\Big)\|\F\z_k\|^2 - \Big(-\frac{b_k\rho_k}{\beta_k}+a_{k+1}\Big)\|\F\z_{k+1}\|^2\\
&+ \frac{b_k\alpha_k}{\beta_k(1-\beta_k)}\inprod{\F\z_{k+1}}{\F\z_{k+1/2}}-\frac{\alpha_k b_k}{\beta_k}\inprod{\F\z_k}{\F\z_{k+1/2}}\\
&+\frac{b_k}{2L_k^2\alpha_k\beta_k(1-\beta_k)}(\|\F\z_{k+1}-\F\z_{k+1/2}\|^2-L_k^2\alpha_k^2\|(1-\beta_k)\F\z_k-\F\z_{k+1/2}\|^2)\\
=&\Big(a_k - \frac{b_k(1-2\beta_k)\rho_k}{\beta_k} - \frac{b_k(1-\beta_k)\alpha_k}{2\beta_k}\Big)\|\F\z_k\|^2 + \Big(\frac{b_k}{2L_k^2\alpha_k\beta_k(1-\beta_k)}+\frac{b_k\rho_k}{\beta_k}-a_{k+1}\Big)\|\F\z_{k+1}\|^2\\
&-\Big(\frac{b_k}{L_k^2\alpha_k\beta_k(1-\beta_k)}-\frac{b_k\alpha_k}{\beta_k(1-\beta_k)}\Big)\inprod{\F\z_{k+1}}{\F\z_{k+1/2}}\\
&+\frac{b_k}{2L_k^2\alpha_k\beta_k(1-\beta_k)}(1-L_k^2\alpha_k^2)\|\F\z_{k+1/2}\|^2\\
=&\Big(\frac{b_k}{2L_k^2\alpha_k(1-\beta_k)\beta_k}+\frac{b_k\rho_k}{\beta_k}-a_{k+1}\Big)\|\F\z_{k+1}\|^2\\
&-\frac{b_k(1-L_k^2\alpha_k^2)}{2L_k^2\alpha_k\beta_k(1-\beta_k)}2\inprod{\F\z_{k+1}}{\F\z_{k+1/2}}\\
&+\frac{b_k(1-L_k^2\alpha_k^2)}{2L_k^2\alpha_k\beta_k(1-\beta_k)}\|\F\z_{k+1/2}\|^2.\\
&\quad(\because a_k = \frac{b_k(1-\beta_k)}{2\beta_k}(\alpha_k+2\rho_k) -b_k\rho_k = \frac{b_k(1-\beta_k)\alpha_k}{2\beta_k}+\frac{b_k(1-2\beta_k)\rho_k}{\beta_k}.)
\end{align*}
Note that the given conditions imply that
\begin{align*}
a_{k+1} &= \frac{b_{k+1}(1-\beta_{k+1})}{2\beta_{k+1}}(\alpha_{k+1}+2\rho_{k+1}) - b_{k+1}\rho_{k+1}\\
&\le\frac{b_{k+1}}{2\beta_k}(\alpha_k+2\rho_k) -b_{k+1}\rho_k\\
&=\frac{b_{k+1}}{2\beta_k}\alpha_k + \frac{b_{k+1}(1-\beta_k)\rho_k}{\beta_k}\\
&=\frac{b_k}{2\beta_k(1-\beta_k)}\alpha_k + \frac{b_k\rho_k}{\beta_k}. \qquad(\because b_{k+1}=\frac{b_k}{1-\beta_k}.)
\end{align*}
Therefore, we get
\begin{align*}
V_k-V_{k+1}&\ge \frac{b_k(1-L_k^2\alpha_k^2)}{2L_k^2\alpha_k\beta_k(1-\beta_k)}(\|\F\z_{k+1}\|^2-2\inprod{\F\z_{k+1}}{\F\z_{k+1/2}}+\|\F\z_{k+1/2}\|^2)\\
&=\frac{b_k(1-L_k^2\alpha_k^2)}{2L_k^2\alpha_k\beta_k(1-\beta_k)}\|\F\z_{k+1}-\F\z_{k+1/2}\|^2\\
&\ge 0.
\end{align*}
Note that $\{\alpha_k\}_{k\ge1}\subseteq \big(0,\frac{1}{L_k}\big]$ and $\{\beta_k\}_{k\ge 1}\subseteq(0,1)$ are the sufficient conditions for $\frac{b_k}{2L_k^2\alpha_k\beta_k(1-\beta_k)}\ge 0$ and $\frac{b_k(1-L_k^2\alpha_k^2)}{2L_k^2\alpha_k\beta_k(1-\beta_k)}\ge 0$ for all $k\ge 1$.
\qed

\subsubsection{Convergence analysis for S-FEG}\label{appx:s_feg_pot}
In this section, we consider the following class of stochastic methods:
\begin{equation}\tag{Class S-FEG}\label{alg:stoc_general}
\begin{aligned}
\z_{k+1/2} &= \z_k + \beta_k (\z_0 - \z_k) - (1-\beta_k)\alpha_k \tF\z_k\\
\z_{k+1} &= \z_k + \beta_k(\z_0-\z_k) - \alpha_k \tF\z_{k+1/2}.
\end{aligned}
\end{equation}


As in the previous section, our analysis relies on the 
potential function, $V_k = a_k\|\F\z_k\|^2 - b_k\inprod{\F\z_k}{\z_0-\z_k}$. 
Although the expectation of the potential function 
is no longer nonincreasing, 
we have a lower bound on $\EE[V_k]-\EE[V_{k+1}]$ 
that
consists of $\sigma_k^2$ and $\sigma_{k+1/2}^2$ 
below.
\begin{lemma}\label{lem:potential_stoc}
	Let $\{\z_k\}_{k\ge 0}$ be the sequence generated by \eqref{alg:stoc_general}
	with
	$\{\alpha_k\}_{k\ge 0}$ and $\{\beta_k\}_{k\ge 0}$ 
	satisfying $\alpha_0\in(0,\infty)$, $\alpha_k\in\big(0,\frac{1}{L}\big]$, $\beta_0 = 1$, $\{\beta_k\}_{k\ge 1}\subseteq(0,1)$ for all $k\ge 1$, and
	\begin{align*}
	\frac{(1-\beta_{k+1})\alpha_{k+1}}{2\beta_{k+1}}
	\le \frac{\alpha_k}{2\beta_k}
	\end{align*}
	for all $k\ge 0$. 
	Assume that $\F$ is $L$-Lipschitz continuous
	and monotone,
	and let $\tF\z_{k/2} = \F\z_{k/2} + \xi_{k/2}$, where $\{\xi_{k/2}\}_{k\ge 0}$ are independent random variables satisfying $\EE[\xi_{k/2}]=0$ and $\EE[\|\xi_{k/2}\|^2]= \sigma_{k/2}^2$ for all $k\ge 0$.
	Then 
	$ 
	V_k = a_k\|\F\z_k\|^2-b_k\inprod{\F\z_k}{\z_0-\z_k}
	$ 
	with $a_0 = \frac{\alpha_0(L^2\alpha_0^2-1)}{2}$, $b_0 = 0$, $b_1 = 1$,
	\begin{align*}
	a_k = \frac{b_k(1-\beta_k)\alpha_k}{2\beta_k} \quad\text{and}\quad
	b_{k+1} = \frac{b_k}{1-\beta_k}
	\end{align*}
	for all $k\ge 1$
	satisfies
	\begin{align*}
	    \EE[V_0]-\EE[V_1] &\ge - \Big(\frac{L^2\alpha_0^3}{2}+L\alpha_0^2\Big)\sigma_0^2 \qquad\text{and}\\
	\EE[V_k]-\EE[V_{k+1}] &\ge -\frac{b_k\alpha_k(1+2L\alpha_k)}{2\beta_k}\Big((1-\beta_k)\sigma_k^2 + \frac{1}{1-\beta_k}\sigma_{k+1/2}^2\Big)
	\end{align*}
	for all $k\ge 1$.
\end{lemma}

We first prove the following lemma
that is used in the proof of Lemma~\ref{lem:potential_stoc}.
\begin{lemma}\label{lem:bounded_exp}
	Let $\tF\z_{k/2} = \F\z_{k/2} + \xi_{k/2}$, where $\{\xi_{k/2}\}_{k\ge 0}$ are independent random variables satisfying $\EE[\xi_{k/2}]=0$ and $\EE[\|\xi_{k/2}\|^2]= \sigma_{k/2}^2$ for all $k\ge 0$.
	Then, for the $L$-Lipschitz continuous 
	and monotone operator $\F$, the sequence
	$\{\z_k\}_{k\ge 0}$ generated by \eqref{alg:stoc_general}
	satisfies
	\begin{align*}
	|\EE[\inprod{\F\z_1}{\tF\z_0-\F\z_0}]| \le L\alpha_0\sigma_0^2
	\end{align*}
	and, for all $k=0,1,\ldots$,
	\begin{align*}
	|\EE[\inprod{\F\z_{k+1/2}}{\tF\z_k-\F\z_k}]|&\le L(1-\beta_k)\alpha_k\sigma_k^2\\
	|\EE[\inprod{\F\z_{k+1}}{\tF\z_{k+1/2}-\F\z_{k+1/2}}]|&\le L\alpha_k\sigma_{k+1/2}^2.
	\end{align*}
\end{lemma}
\begin{proof}
    We have that
	\begin{align*}
	|\EE[\inprod{\F\z_1}{\tF\z_0-\F\z_0}]|
	&=|\EE[\inprod{\F\z_1-\F(\z_0-\alpha_0\F\z_0)}{\tF\z_0-\F\z_0}]|\\
	&\le\EE[\|\F\z_1-\F(\z_0-\alpha_0\F\z_0)\|\,\|\tF\z_0-\F\z_0\|]\\
	&\le\EE[L\|\z_1-(\z_0-\alpha_0\F\z_0)\|\,\|\tF\z_0-\F\z_0\|]\\
	&=\EE[L\alpha_0\|\tF\z_0-\F\z_0\|^2]\\
	&=L\alpha_0\sigma_0^2,
	\end{align*}
	where the first equality uses
	the assumption that $\xi_0 = \tF\z_0 - \F\z_0$
	is an independent random variable 
	with $\EE[\xi_0]=0$.
	Similarly, we have that
	\begin{align*}
	|\EE[\inprod{\F\z_{k+1/2}}{\tF\z_k-\F\z_k}]|
	&= |\EE[\inprod{\F\z_{k+1/2}-\F(\z_k+\beta_k(\z_0-\z_k)-(1-\beta_k)\alpha_k\F\z_k)}{\tF\z_k-\F\z_k}]|\\
	&\le\EE[\|\F\z_{k+1/2}-\F(\z_k+\beta_k(\z_0-\z_k)-(1-\beta_k)\alpha_k\F\z_k)\|\,\|\tF\z_k-\F\z_k\|]\\
	&\le\EE[L\|\z_{k+1/2}-(\z_k+\beta_k(\z_0-\z_k)-(1-\beta_k)\alpha_k\F\z_k)\|\,\|\tF\z_k-\F\z_k\|]\\
	&=\EE[L(1-\beta_k)\alpha_k\|\tF\z_k-\F\z_k\|^2]\\
	&=L(1-\beta_k)\alpha_k\sigma_k^2,
	\end{align*}
	and
	\begin{align*}
	&|\EE[\inprod{\F\z_{k+1}}{\tF\z_{k+1/2}-\F\z_{k+1/2}}]|\\
	&=|\EE[\inprod{\F\z_{k+1}-\F(\z_k+\beta_k(\z_0-\z_k)-\alpha_k\F\z_{k+1/2})}{\tF\z_{k+1/2}-\F\z_{k+1/2}}]|\\
	&\le\EE[\|\F\z_{k+1}-\F(\z_k+\beta_k(\z_0-\z_k)-\alpha_k\F\z_{k+1/2})\|\,\|\tF\z_{k+1/2}-\F\z_{k+1/2}\|]\\
	&\le\EE[L\|\z_{k+1}-(\z_k+\beta_k(\z_0-\z_k)-\alpha_k\F\z_{k+1/2})\|\,\|\tF\z_{k+1/2}-\F\z_{k+1/2}\|]\\
	&=\EE[L\alpha_k\|\tF\z_{k+1/2}-\F\z_{k+1/2}\|^2]\\
	&=L\alpha_k\sigma_{k+1/2}^2.
	\end{align*}
\end{proof}

\textit{Proof of Lemma~\ref{lem:potential_stoc}.}
First, for $k=0$, note that
	\begin{align}
	V_1 &= a_1\|\F\z_1\|^2 - b_1\inprod{\F\z_1}{\z_0-\z_1}\nonumber\\
	&=a_1\|\F\z_1\|^2 - \alpha_0b_1\inprod{\F\z_1}{\tF\z_0}\nonumber\\
	&=\frac{b_1(1-\beta_1)\alpha_1}{2\beta_1}\|\F\z_1\|^2 - \alpha_0\inprod{\F\z_1}{\tF\z_0}\nonumber\\
	&\le \frac{b_1\alpha_0}{2\beta_0}\|\F\z_1\|^2 - \alpha_0\inprod{\F\z_1}{\tF\z_0}\nonumber\\
	&= \frac{\alpha_0}{2}\|\F\z_1\|^2 - \alpha_0\inprod{\F\z_1}{\tF\z_0}.\label{eq:init_potential_gap_stoc_1}
	\end{align}
	By the given condition, we get
	\begin{align}\label{eq:init_potential_gap_stoc_2}
	0\le L_0^2\|\z_1-\z_0\|^2-\|\F\z_1-\F\z_0\|^2 = L_0^2\alpha_0^2\|\tF\z_0\|^2 - \|\F\z_1-\F\z_0\|^2.
	\end{align}
	The sum of \eqref{eq:init_potential_gap_stoc_1} and \eqref{eq:init_potential_gap_stoc_2} with multiplying factor $\frac{\alpha_0}{2}$ yields
	\begin{align*}
	V_1&\le \frac{\alpha_0}{2}\|\F\z_1\|^2 - \alpha_0\inprod{\F\z_1}{\tF\z_0} 
	+ \frac{\alpha_0}{2}(L_0^2\alpha_0^2\|\tF\z_0\|^2 - \|\F\z_1-\F\z_0\|^2)\\
	&= \frac{\alpha_0}{2}(L^2\alpha_0^2\|\tF\z_0\|^2-\|\F\z_0\|^2)-\alpha_0\inprod{\F\z_1}{\tF\z_0-\F\z_0}.
	\end{align*}
	Hence, we get
	\begin{align*}
	V_0-V_1\ge & \frac{\alpha_0(L^2\alpha_0^2-1)}{2}\|\F\z_0\|^2
	- \frac{\alpha_0}{2}(L^2\alpha_0^2\|\tF\z_0\|^2-\|\F\z_0\|^2)
	+ \alpha_0\inprod{\F\z_1}{\tF\z_0-\F\z_0}\\
	=& \frac{L^2\alpha_0^3}{2}(\|\F\z_0\|^2-\|\tF\z_0\|^2)
	+ \alpha_0\inprod{\F\z_1}{\tF\z_0-\F\z_0}.
	\end{align*}
	By taking expectation on the both sides,
	\begin{align*}
	\EE[V_0]-\EE[V_1] \ge& -\frac{L^2\alpha_0^3}{2}\sigma_0^2
	+ \alpha_0\EE[\inprod{\F\z_1}{\tF\z_0-\F\z_0}]\\
	\ge& - \Big(\frac{L^2\alpha_0^3}{2}+L\alpha_0^2\Big)\sigma_0^2
	\end{align*}
	where 
	the first inequality uses the fact
	$\EE[\|\F\z_0\|^2-\|\tF\z_0\|^2] 
	= -\EE[\|\F\z_0 - \tF\z_0\|^2]$,
	and
	the last inequality follows from Lemma~\ref{lem:bounded_exp}.
	Next, for $k\ge 1$,
	here we note the following relations for later use:
    \begin{align*}
    \z_{k+1}-\z_k &= \frac{\beta_k}{1-\beta_k}(\z_0-\z_{k+1}) - \frac{\alpha_k}{1-\beta_k}\tF\z_{k+1/2},\\
    \z_{k+1}-\z_k &= \beta_k(\z_0-\z_k) - \alpha_k\tF\z_{k+1/2}, \text{ and}\\
    \z_{k+1}-\z_{k+1/2}&= \alpha_k((1-\beta_k)\tF\z_k - \tF\z_{k+1/2}).
    \end{align*}
	Then, by the given condition, we have
	\begingroup
	\allowdisplaybreaks
	\begin{align}
	V_k - V_{k+1}\ge& V_k - V_{k+1} - \frac{b_k}{\beta_k}\inprod{\F\z_{k+1}-\F\z_k}{\z_{k+1}-\z_k}\nonumber\\
	=&V_k - V_{k+1} - \frac{b_k}{\beta_k}\inprod{\F\z_{k+1}}{\z_{k+1}-\z_k} + \frac{b_k}{\beta_k}\inprod{\F\z_k}{\z_{k+1}-\z_k}\nonumber\\
	=&(a_k\|\F\z_k\|^2-b_k\inprod{\F\z_k}{\z_0-\z_k})-(a_{k+1}\|\F\z_{k+1}\|^2-b_{k+1}\inprod{\F\z_{k+1}}{\z_0-\z_{k+1}})\nonumber\\
	&-\frac{b_k}{\beta_k}\Inprod{\F\z_{k+1}}{\frac{\beta_k}{1-\beta_k}(\z_0-\z_{k+1})-\frac{\alpha_k}{1-\beta_k}\tF\z_{k+1/2}}\nonumber\\
	&+\frac{b_k}{\beta_k}\inprod{\F\z_k}{\beta_k(\z_0-\z_k)-\alpha_k\tF\z_{k+1/2}}\nonumber\\
	=&a_k\|\F\z_k\|^2 - a_{k+1}\|\F\z_{k+1}\|^2\nonumber\\
	&+\Big(b_{k+1}-\frac{b_k}{1-\beta_k}\Big)\inprod{\F\z_{k+1}}{\z_0-\z_{k+1}}+\frac{b_k\alpha_k}{\beta_k(1-\beta_k)}\inprod{\F\z_{k+1}}{\tF\z_{k+1/2}}\nonumber\\
	&-\frac{\alpha_k b_k}{\beta_k}\inprod{\F\z_k}{\tF\z_{k+1/2}}\nonumber\\
	=&a_k\|\F\z_k\|^2 - a_{k+1}\|\F\z_{k+1}\|^2\nonumber\\
	&+\frac{b_k\alpha_k}{\beta_k(1-\beta_k)}\inprod{\F\z_{k+1}}{\tF\z_{k+1/2}}-\frac{\alpha_k b_k}{\beta_k}\inprod{\F\z_k}{\tF\z_{k+1/2}}. \qquad(\because b_{k+1}=\frac{b_k}{1-\beta_k}.)\label{eq:potential_gap_stoc_1}
	\end{align}
	\endgroup
	By the given condition, we get
	\begin{align}\label{eq:potential_gap_stoc_2}
	0&\ge \|\F\z_{k+1}-\F\z_{k+1/2}\|^2 - L_k^2\|\z_{k+1}-\z_{k+1/2}\|^2 \nonumber\\
	&=\|\F\z_{k+1}-\F\z_{k+1/2}\|^2 - L_k^2\alpha_k^2\|(1-\beta_k)\tF\z_k-\tF\z_{k+1/2}\|^2.
	\end{align}
	Hence, the sum of \eqref{eq:potential_gap_stoc_1} and \eqref{eq:potential_gap_stoc_2} with multiplying factor $\frac{b_k}{2L_k^2\alpha_k\beta_k(1-\beta_k)}$ yields
	\begin{align*}
	V_k-V_{k+1}\ge& a_k\|\F\z_k\|^2 - a_{k+1}\|\F\z_{k+1}\|^2\\
	&+\frac{b_k\alpha_k}{\beta_k(1-\beta_k)}\inprod{\F\z_{k+1}}{\tF\z_{k+1/2}}-\frac{\alpha_k b_k}{\beta_k}\inprod{\F\z_k}{\tF\z_{k+1/2}}\\
	&+\frac{b_k}{2L^2\alpha_k\beta_k(1-\beta_k)}(\|\F\z_{k+1}-\F\z_{k+1/2}\|^2-L_k^2\alpha_k^2\|(1-\beta_k)\tF\z_k-\tF\z_{k+1/2}\|^2)\\
	=&\Big(a_k\|\F\z_k\|^2 - \frac{b_k(1-\beta_k)\alpha_k}{2\beta_k}\|\tF\z_k\|^2\Big)
	+ \Big(\frac{b_k}{2L^2\alpha_k\beta_k(1-\beta_k)}-a_{k+1}\Big)\|\F\z_{k+1}\|^2\\
	&+ \frac{b_k}{2L^2\alpha_k\beta_k(1-\beta_k)}\|\F\z_{k+1/2}\|^2
	- \frac{b_k\alpha_k}{2\beta_k(1-\beta_k)}\|\tF\z_{k+1/2}\|^2\\
	&-\frac{b_k}{L^2\alpha_k\beta_k(1-\beta_k)}\inprod{\F\z_{k+1}}{\F\z_{k+1/2}}
	+ \frac{b_k\alpha_k}{\beta_k(1-\beta_k)}\inprod{\F\z_{k+1}}{\tF\z_{k+1/2}}\\
	&+\frac{b_k\alpha_k}{\beta_k}\inprod{\tF\z_k-\F\z_k}{\tF\z_{k+1/2}}
	\end{align*}
	By the given conditions, we get $a_k = \frac{b_k(1-\beta_k)\alpha_k}{2\beta_k}$ and
	\begin{align*}
	a_{k+1} = \frac{b_{k+1}(1-\beta_{k+1})\alpha_{k+1}}{2\beta_{k+1}}
	\le\frac{b_{k+1}\alpha_k}{2\beta_k}
	=\frac{b_k\alpha_k}{2\beta_k(1-\beta_k)}. \qquad(\because b_{k+1}=\frac{b_k}{1-\beta_k}.)
	\end{align*}
	Therefore, we get
	\begingroup
	\allowdisplaybreaks
	\begin{align*}
	V_k-V_{k+1}\ge& \frac{b_k(1-\beta_k)\alpha_k}{2\beta_k}(\|\F\z_k\|^2-\|\tF\z_k\|^2)
	+ \frac{b_k(1-L^2\alpha_k^2)}{2L^2\alpha_k\beta_k(1-\beta_k)}\|\F\z_{k+1}\|^2\\
	&+ \frac{b_k}{2L^2\alpha_k\beta_k(1-\beta_k)}\|\F\z_{k+1/2}\|^2
	- \frac{b_k\alpha_k}{2\beta_k(1-\beta_k)}\|\tF\z_{k+1/2}\|^2\\
	&-\frac{b_k}{L^2\alpha_k\beta_k(1-\beta_k)}\inprod{\F\z_{k+1}}{\F\z_{k+1/2}}
	+ \frac{b_k\alpha_k}{\beta_k(1-\beta_k)}\inprod{\F\z_{k+1}}{\tF\z_{k+1/2}}\\
	&+\frac{b_k\alpha_k}{\beta_k}\inprod{\tF\z_k-\F\z_k}{\tF\z_{k+1/2}}\\
	=& \frac{b_k(1-\beta_k)\alpha_k}{2\beta_k}(\|\F\z_k\|^2-\|\tF\z_k\|^2)
	+ \frac{b_k(1-L^2\alpha_k^2)}{2L^2\alpha_k\beta_k(1-\beta_k)}\|\F\z_{k+1}\|^2\\
	&+ \frac{b_k(1-L^2\alpha_k^2)}{2L^2\alpha_k\beta_k(1-\beta_k)}\|\F\z_{k+1/2}\|^2
	+ \frac{b_k\alpha_k}{2\beta_k(1-\beta_k)}(\|\F\z_{k+1/2}\|^2-\|\tF\z_{k+1/2}\|^2)\\
	&-\frac{b_k(1-L^2\alpha_k^2)}{L^2\alpha_k\beta_k(1-\beta_k)}\inprod{\F\z_{k+1}}{\F\z_{k+1/2}}
	+ \frac{b_k\alpha_k}{\beta_k(1-\beta_k)}\inprod{\F\z_{k+1}}{\tF\z_{k+1/2}-\F\z_{k+1/2}}\\
	&+\frac{b_k\alpha_k}{\beta_k}\inprod{\tF\z_k-\F\z_k}{\tF\z_{k+1/2}}\\
	=& \frac{b_k(1-\beta_k)\alpha_k}{2\beta_k}(\|\F\z_k\|^2-\|\tF\z_k\|^2)
	+ \frac{b_k\alpha_k}{2\beta_k(1-\beta_k)}(\|\F\z_{k+1/2}\|^2-\|\tF\z_{k+1/2}\|^2)
	\\
	&+ \frac{b_k(1-L^2\alpha_k^2)}{2L^2\alpha_k\beta_k(1-\beta_k)}\|\F\z_{k+1}-\F\z_{k+1/2}\|^2\\
	&+\frac{b_k\alpha_k}{\beta_k(1-\beta_k)}\inprod{\F\z_{k+1}}{\tF\z_{k+1/2}-\F\z_{k+1/2}}
	+\frac{b_k\alpha_k}{\beta_k}\inprod{\tF\z_k-\F\z_k}{\tF\z_{k+1/2}}\\
	\ge& \frac{b_k(1-\beta_k)\alpha_k}{2\beta_k}(\|\F\z_k\|^2-\|\tF\z_k\|^2)
	+ \frac{b_k\alpha_k}{2\beta_k(1-\beta_k)}(\|\F\z_{k+1/2}\|^2-\|\tF\z_{k+1/2}\|^2)
	\\
	&+\frac{b_k\alpha_k}{\beta_k(1-\beta_k)}\inprod{\F\z_{k+1}}{\tF\z_{k+1/2}-\F\z_{k+1/2}}
	+\frac{b_k\alpha_k}{\beta_k}\inprod{\tF\z_k-\F\z_k}{\tF\z_{k+1/2}}.
	\end{align*}
	\endgroup
	Then by taking expectation on the both sides and using the fact $\EE[\inprod{\tF\z_k-\F\z_k}{\tF\z_{k+1/2}}] = \EE[\inprod{\tF\z_k-\F\z_k}{\F\z_{k+1/2}}]$, we get
	\begin{align*}
	\EE[V_k]-\EE[V_{k+1}] \ge & -\frac{b_k(1-\beta_k)\alpha_k}{2\beta_k}\sigma_k^2
	-\frac{b_k\alpha_k}{2\beta_k(1-\beta_k)}\sigma_{k+1/2}^2\\
	&+\frac{b_k\alpha_k}{\beta_k(1-\beta_k)}\EE[\inprod{\F\z_{k+1}}{\tF\z_{k+1/2}-\F\z_{k+1/2}}]
	+\frac{b_k\alpha_k}{\beta_k}\EE[\inprod{\tF\z_k-\F\z_k}{\F\z_{k+1/2}}]\\
	\ge&-\frac{b_k\alpha_k(1+2L\alpha_k)}{2\beta_k}\Big((1-\beta_k)\sigma_k^2 + \frac{1}{1-\beta_k}\sigma_{k+1/2}^2\Big)
	\end{align*}
	where the last inequality follows from Lemma~\ref{lem:bounded_exp}.
\qed

\subsubsection{Proof of Theorem~\ref{thm:convergence_stoc}}\label{appx:s_feg_thm}
Note that S-FEG is equivalent to \eqref{alg:stoc_general} with $\alpha_k = \frac{1}{L}$ and $\beta_k = \frac{1}{k+1}$.
It is straightforward to verify that the given $\{\alpha_k\}_{k\ge 0}$ and $\{\beta_k\}_{k\ge 0}$ satisfy the conditions in Lemma~\ref{lem:potential_stoc} for all $k\ge 0$. By noting that
\begin{gather*}
a_k = \frac{b_k(1-\beta_k)\alpha_k}{2\beta_k} = \frac{k^2}{2L} \qquad\text{and}\\
b_k = \frac{1}{1-\beta_{k-1}}b_{k-1} = \Big(\prod_{i=1}^{k-1}\frac{1}{1-\beta_i}\Big)b_1 = k,
\end{gather*}
Lemma~\ref{lem:potential_stoc} implies that
\begin{align*}
\EE[V_k]&=\EE[V_k-V_0]\\
&=\sum_{l=0}^{k-1}(\EE[V_{l+1}]-\EE[V_l])\\
&\le \Big(\frac{L^2\alpha_0^3}{2}+L\alpha_0^2\Big)\sigma_0^2 + \sum_{l=1}^{k-1} \frac{b_l\alpha_l(1+2L\alpha_l)}{2\beta_l}\Big((1-\beta_l)\sigma_l^2 + \frac{1}{1-\beta_l}\sigma_{l+1/2}^2\Big)\\
&= \frac{3}{2L}\sigma_0^2 + \sum_{l=1}^{k-1}\frac{3}{2L}(l^2\sigma_l^2 + (l+1)^2\sigma_{l+1/2}^2)\\
&\overset{\text{let}}{=}\sigma_{\text{total}}^2.
\end{align*}
Therefore, noting that $\EE[V_k] = \EE\Big[\frac{k^2}{2L}\|\F\z_k\|^2 - k\inprod{\F\z_k}{\z_0-\z_k}\Big]$, we get
\begin{align*}
\EE\Big[\frac{k^2}{2L}\|\F\z_k\|^2\Big]&\le \EE[k\inprod{\F\z_k}{\z_0-\z_k}] + \sigma_{\text{total}}^2\\
&= \EE[k\inprod{\F\z_k}{\z_0-\z_*} + k\inprod{\F\z_k}{\z_*-\z_k}] + \sigma_{\text{total}}^2\\
&= \EE[k\inprod{\F\z_k}{\z_0-\z_*}] + \sigma_{\text{total}}^2 \qquad(\because \text{monotonicity of $\F$})\\
&\le \EE\Big[\frac{k^2}{4L}\|\F\z_k\|^2 + L\|\z_0-\z_*\|^2\Big] + \sigma_{\text{total}}^2. \qquad(\because \text{Young's inequality}.)
\end{align*}
As a result, we get $\EE\Big[\frac{k^2}{4L}\|\F\z_k\|^2\Big]\le L\|\z_0-\z_*\|^2 + \frac{3}{2L}\left[\sigma_0^2 + \sum_{l=1}^{k-1}(l^2\sigma_l^2 + (l+1)^2\sigma_{l+1/2}^2)\right]$ and, by dividing the both sides by $\frac{k^2}{4L}$, we get
\begin{align*}
    \EE[\|\F\z_k\|^2]\le \frac{4L^2\|\z_0-\z_*\|^2}{k^2} + \frac{6}{k^2}\left[\sigma_0^2 + \sum_{l=1}^{k-1}(l^2\sigma_l^2 + (l+1)^2\sigma_{l+1/2}^2)\right].
\end{align*}
In addition, if 
$\sigma_0^2 \le \frac{\epsilon}{6}$,
$\sigma_k^2\le \frac{\epsilon}{6k}$ and $\sigma_{k+1/2}^2\le \frac{\epsilon}{6(k+1)}$ for all $k\ge 1$, then we have
\begin{align*}
    \EE[\|\F\z_k\|^2] &\le \frac{4L^2\|\z_0-\z_*\|^2}{k^2} + \frac{6}{k^2}\left[\frac{\epsilon}{6} + \sum_{l=1}^{k-1}\Big(\frac{\epsilon l}{6} + \frac{\epsilon(l+1)}{6}\Big)\right]\\
    &= \frac{4L^2\|\z_0-\z_*\|^2}{k^2} + \frac{\epsilon}{k^2}\sum_{l=0}^{k-1}(2l+1)\\
    &= \frac{4L^2\|\z_0-\z_*\|^2}{k^2} + \frac{\epsilon}{k^2}\sum_{l=0}^{k-1}((l+1)^2-l^2)\\
    &= \frac{4L^2\|\z_0-\z_*\|^2}{k^2} + \epsilon.
\end{align*}
\qed

\comment{
\subsubsection{Proof of Theorem~\ref{thm:convergence_stoc}}
Recall that by taking $\beta_k = \frac{1}{k+1}$ and $\alpha_k = \frac{1}{L}$ in \eqref{alg:stoc_general}, we get FEG.
It is straightforward to verify that the given $\{\alpha_k\}_{k\ge 0}$ and $\{\beta_k\}_{k\ge 0}$ satisfy the conditions in Lemma~\ref{lem:potential_stoc} for all $k\ge 0$. By noting that
\begin{gather*}
b_k = \frac{1}{1-\beta_{k-1}}b_{k-1} = \Big(\prod_{i=1}^{k-1}\frac{1}{1-\beta_i}\Big)b_1 = k \qquad\text{and}\\
a_k = \frac{b_k(1-\beta_k)\alpha_k}{2\beta_k} = \frac{k^2}{2L},
\end{gather*}
Lemma~\ref{lem:potential_stoc} implies that
\begin{align*}
\EE[V_k]&=\EE[V_k-V_0]\\
&\le \Big(\frac{L^2\alpha_0^3}{2}+L\alpha_0^2\Big)\sigma_0^2 + \sum_{l=1}^{k-1} \frac{b_l\alpha_l(1+2L\alpha_l)}{2\beta_l}\Big((1-\beta_l)\sigma_l^2 + \frac{1}{1-\beta_l}\sigma_{l+1/2}^2\Big)\\
&= \frac{3}{2L}\sigma_0^2 + \sum_{l=1}^{k-1}\frac{3}{2L}(l^2\sigma_l^2 + (l+1)^2\sigma_{l+1/2}^2)\\
&\le \frac{\epsilon}{4L} + \sum_{l=1}^{k-1}\frac{\epsilon}{4L}(2l+1)\\
&= \frac{k^2\epsilon}{4L}.
\end{align*}
Therefore, noting that $\EE[V_k] = \EE\Big[\frac{k^2}{2L}\|\F\z_k\|^2 - k\inprod{\F\z_k}{\z_0-\z_k}\Big]$,
\begin{align*}
\EE\Big[\frac{k^2}{2L}\|\F\z_k\|^2\Big]&\le \EE[k\inprod{\F\z_k}{\z_0-\z_k}] + \frac{k^2\epsilon}{4L}\\
&= \EE[k\inprod{\F\z_k}{\z_0-\z_*} + k\inprod{\F\z_k}{\z_*-\z_k}] + \frac{k^2\epsilon}{4L}\\
&= \EE[k\inprod{\F\z_k}{\z_0-\z_*}] + \frac{k^2\epsilon}{4L} \qquad(\because \text{monotonicity of $\F$})\\
&\le \EE\Big[\frac{k^2}{4L}\|\F\z_k\|^2 + L\|\z_0-\z_*\|^2\Big] + \frac{k^2\epsilon}{4L} \qquad(\because \text{Young's inequality})
\end{align*}
As a result, we get $\EE\Big[\frac{k^2}{4L}\|\F\z_k\|^2\Big]\le L\|\z_0-\z_*\|^2 + \frac{k^2\epsilon}{4L}$ and
the desired result follows directly by dividing both sides by $\frac{k^2}{4L}$.
\qed
}
\end{document}

%% file: def,command.tex
\input def,gen

\input def,math

\usepackage{jf-accent}
\usepackage{jf-fun}
\usepackage{jf-names}

%% file: def,gen.tex
\providecommand{\ie}		{\emph{i.e\@.}\xspace}
\providecommand{\eg}		{\emph{e.g\@.}\xspace}

\providecommand{\myurl}[1][]	{\texttt{web.eecs.umich.edu/$\sim$fessler#1}\xspace}

\providecommand{\onweb}[1]	{Available from \myurl.}

\long\def\comment#1{}

\providecommand{\bcent}		{\begin{center}}
\providecommand{\ecent}		{\end{center}}
\providecommand{\benum}		{\begin{enumerate}}
\providecommand{\eenum}		{\end{enumerate}}
\providecommand{\bitem}		{\begin{itemize}}
\providecommand{\eitem}		{\end{itemize}}
\providecommand{\bvers}		{\begin{verse}}
\providecommand{\evers}		{\end{verse}}
\providecommand{\btab}		{\begin{tabbing}}	
\providecommand{\etab}		{\end{tabbing}}

\newcounter{blist}
\providecommand{\blistmark}	{\makebox[0pt]{$\bullet$}}
\providecommand{\blistitemsep}	{0pt}
\providecommand{\blist}[1][]	{%
\begin{list}{\blistmark}{%
\usecounter{blist}%
\setlength{\itemsep}{\blistitemsep}%
\setlength{\parsep}{0pt}%
\setlength{\parskip}{0pt}%
\setlength{\partopsep}{0pt}%
\setlength{\topsep}{0pt}%
\setlength{\leftmargin}{1.2em}%
\setlength{\labelsep}{0.5\leftmargin}
\setlength{\labelwidth}{0em}%
#1}
}
\providecommand{\elist}		{\end{list}}

\providecommand{\blistitemsep}	{0pt}
\providecommand{\bjfenum}[1][]	{%
\begin{list}{\bcolor{\arabic{blist}.} }{%
\usecounter{blist}%
\setlength{\itemsep}{\blistitemsep}%
\setlength{\parsep}{0pt}%
\setlength{\parskip}{0pt}%
\setlength{\partopsep}{0pt}%
\setlength{\topsep}{0pt}%
\setlength{\leftmargin}{0.0em}%
\setlength{\labelsep}{1.0\leftmargin}
\setlength{\labelwidth}{0pt}%
#1}
}

\newcounter{blistAlph}
\providecommand{\blistAlph}[1][]
{\begin{list}{\makebox[0pt][l]{\Alph{blistAlph}.}}{%
\usecounter{blistAlph}%
\setlength{\itemsep}{0pt}\setlength{\parsep}{0pt}%
\setlength{\parskip}{0pt}\setlength{\partopsep}{0pt}%
\setlength{\topsep}{0pt}%
\setlength{\leftmargin}{1.2em}%
\setlength{\labelsep}{1.0\leftmargin}
\setlength{\labelwidth}{0.0\leftmargin}#1}%
}

\newcounter{blistRoman}
\providecommand{\blistRoman}[1][]
{\begin{list}{\Roman{blistRoman}.}{%
\usecounter{blistRoman}%
\setlength{\itemsep}{0.5em}\setlength{\parsep}{0pt}%
\setlength{\parskip}{0pt}\setlength{\partopsep}{0pt}%
\setlength{\topsep}{0pt}%
\setlength{\leftmargin}{4em}%
\setlength{\labelsep}{0.4\leftmargin}
\setlength{\labelwidth}{0.6\leftmargin}#1}%
}

%% file: def,math.tex
\usepackage{bbm} 
\providecommand{\jfbbm}[1]	{\xmath{\mathbbm{#1}}} 
\providecommand{\qed}[1][0pt]	{\hfill\raisebox{#1}{\inmath{\Box}}} 

\providecommand{\reals}		{\jfbbm{R}}

\providecommand{\inprod}[2]	{\xmath{\mathop{\langle #1,\, #2 \rangle}\nolimits}}
\providecommand{\Inprod}[2]	{\xmath{\left\langle #1,\ #2 \right\rangle}}

\let\equivsave\equiv
\def\equiv{\xmath{\equivsave}}

\providecommand{\ba}[1]		{\left[ \begin{array}{#1}}
\providecommand{\ea}		{\end{array} \right]}
\providecommand{\be}		{\begin{equation}}
\providecommand{\ee}[1]		{\label{#1}\end{equation}}
\providecommand{\bea}		{\begin{eqnarray}}
\providecommand{\eea}[1]	{\label{#1}\end{eqnarray}}
\providecommand{\beas}		{\begin{eqnarray*}}
\providecommand{\eeas}		{\end{eqnarray*}}
\providecommand{\beals}[1][1]	{\begin{alignat*}{#1}}	
\providecommand{\eeals}		{\end{alignat*}}

\providecommand{\berr}[2]{
\bgroup
\renewcommand{\theequation}{#1}
\be
#2
\ee{e,#1}
\egroup
\ignorespaces
}

\providecommand{\bearr}[2]{
\bgroup
\renewcommand{\theequation}{#1}
\bea
#2
\eea{e,#1}
\egroup
\ignorespaces
}

\providecommand{\inmath}	{\ensuremath}
\providecommand{\xmath}[1]	{\inmath{#1}\xspace}
\providecommand{\bmath}[1]	{\xmath{\bm{#1}}}	

\providecommand{\paren}[1]	{\xmath{\left(#1\right)}}

\providecommand{\xfrac}[2]	{\xmath{\frac{#1}{#2}}}

\providecommand{\pfrac}[3][]	{\xmath{\!\paren{#1\frac{#2}{#3}}}}

\providecommand{\sinc}		{\Ofun{\mathrm{sinc}}}

